\documentclass[9pt,twocolumn,twoside]{pnas-new}

\templatetype{pnasresearcharticle} 

\usepackage{svg}
\usepackage{bm}
\usepackage{mathtools}
\usepackage{amsmath}
\usepackage{wasysym}
\usepackage{mathdots}
\DeclareRobustCommand
  \Compactcdots{\mathinner{\cdotp\mkern-3mu\cdotp\mkern-3mu\cdotp}}
\algrenewcommand\algorithmicrequire{\textbf{Input:}}
\algrenewcommand\algorithmicensure{\textbf{Output:}}
\DeclareRobustCommand
  \Compactcdots{\mathinner{\cdotp\mkern-3mu\cdotp\mkern-3mu\cdotp}}
\algrenewcommand\algorithmicrequire{\textbf{Input:}}
\algrenewcommand\algorithmicensure{\textbf{Output:}}
\usepackage{amsthm}
\usepackage{amsmath}
\newtheorem{theorem}{Theorem}[subsection]

\theoremstyle{remark}

\newtheorem{definition}{Definition}[subsection]

\usepackage{mathtools}
\usepackage{amssymb}

\usepackage {mathtools}
\usepackage{amsmath}
\usepackage{algpseudocode}
\usepackage{algorithm}
\usepackage{CJKutf8}
\algnewcommand\algorithmicforeach{\textbf{for each}}
\algdef{S}[FOR]{ForEach}[1]{\algorithmicforeach\ #1\ \algorithmicdo}
\DeclareRobustCommand
  \Compactcdots{\mathinner{\cdotp\mkern-3mu\cdotp\mkern-3mu\cdotp}}

\title{Robust estimations from distribution structures: I. Mean}
\author[a,b,c,1]{Li Tuobang}

\leadauthor{Li}

\significancestatement{In 1964, van Zwet introduced the convex transformation order for comparing the skewness of two distributions. This paradigm shift played a fundamental role in defining robust measures of distributions, from spread to kurtosis. Here, instead of examining the stochastic ordering between two distributions, the orderliness of quantile combinations within a distribution is investigated. By classifying distributions through the signs of derivatives, two series of sophisticated robust mean estimators are deduced. Nearly all common nonparametric robust location estimators are found to be special cases thereof.}

\authorcontributions{L.T. designed research, performed research, analyzed data, and wrote the paper.}
\authordeclaration{The author declares no competing interest.}
\correspondingauthor{\textsuperscript{1}To whom correspondence should be addressed. E-mail: tuobang@biomathematics.org}

\begin{abstract}
As the most fundamental problem in statistics, robust location estimation has many prominent solutions, such as the trimmed mean, Winsorized mean, Hodges–Lehmann estimator, Huber $M$-estimator, and median of means. Recent studies suggest that their maximum biases concerning the mean can be quite different, but the underlying mechanisms largely remain unclear. This study exploited a semiparametric method to classify distributions by the asymptotic orderliness of quantile combinations with varying breakdown points, showing their interrelations and connections to parametric distributions. Further deductions explain why the Winsorized mean typically has smaller biases compared to the trimmed mean; two sequences of semiparametric robust mean estimators emerge, particularly highlighting the superiority of the median Hodges–Lehmann mean. This article sheds light on the understanding of the common nature of probability distributions.

\end{abstract}

\dates{This manuscript was compiled on \today}

\begin{document}

\maketitle
\thispagestyle{firststyle}
\ifthenelse{\boolean{shortarticle}}{\ifthenelse{\boolean{singlecolumn}}{\abscontentformatted}{\abscontent}}{}


\dropcap{I}n 1823, Gauss \cite{gauss1823theoria} proved that for any unimodal distribution, $\left|m-\mu\right|\le\sqrt{\frac{3}{4}}\omega$ and $\sigma\leq \omega\leq 2\sigma$, where $\mu$ is the population mean, $m$ is the population median, $\omega$ is the root mean square deviation from the mode, and $\sigma$ is the population standard deviation. This pioneering work revealed that, the potential bias of the median with respect to the mean is bounded in units of a scale parameter under certain assumptions. In 2018, Li, Shao, Wang, and Yang \cite{li2018worst} proved the bias bound of any quantile for arbitrary continuous distributions with finite second moments. Bernard, Kazzi, and Vanduffel (2020) \cite{bernard2020range} further refined these bounds for unimodal distributions with finite second moments and extended to the bounds of symmetric quantile averages. They showed that $m$ has the smallest maximum distance to $\mu$ among all symmetric quantile averages ($\text{SQA}_{\epsilon}$). Daniell, in 1920, \cite{daniell1920observations} analyzed a class of estimators, linear combinations of order statistics, and identified that the $\epsilon$-symmetric trimmed mean ($\text{STM}_{\epsilon}$) belongs to this class. Another popular choice, the $\epsilon$-symmetric Winsorized mean ($\text{SWM}_{\epsilon}$), named after Winsor and introduced by Tukey \cite{tukey1960survey} and Dixon \cite{dixon1960simplified} in 1960, is also an $L$-estimator. Bieniek (2016) derived exact bias upper bounds of the Winsorized mean based on Danielak and Rychlik's work (2003) on the trimmed mean for any distribution with a finite second moment and confirmed that the former is smaller than the latter \cite{danielak2003theory,bieniek2016comparison}. Oliveira and Orenstein (2019) and 
Lugosi and Mendelson (2021) \cite{oliveira2019sub,trimmedMendelson} derived the concentration bound of the trimmed mean. In 1963, Hodges and Lehmann \cite{hodges1963estimates} proposed a class of nonparametric location estimators based on rank tests and, from the Wilcoxon signed-rank statistic \cite{wilcoxon1945individual}, deduced the median of pairwise means as a robust location estimator for a symmetric population. Both $L$-statistics and $R$-statistics achieve robustness essentially by removing a certain proportion of extreme values, therefore, they have predefined breakdown points \cite{hampel1968contributions}. In 1964, Huber \cite{huber1964robust} generalized maximum likelihood estimation to the minimization of the sum of a specific loss function, which measures the residuals between the data points and the model's parameters. Some $L$-estimators are also $M$-estimators, e.g., the sample mean is an $M$-estimator with a squared error loss function, the sample median is an $M$-estimator with an absolute error loss function \cite{huber1964robust}. The Huber $M$-estimator is obtained by applying the Huber loss function that combines elements of both squared error and absolute error to achieve robustness against gross errors and high efficiency for contaminated Gaussian distributions \cite{huber1964robust}. Sun, Zhou, and Fan (2020) examined the concentration bounds of the Huber $M$-estimator \cite{sun2020adaptive}. In 2012, Catoni proposed an $M$-estimator for heavy-tailed samples
with finite variance \cite{catoni2012challenging}. Xu (2021) \cite{XuCatoni} proposed a generalized Catoni $M$-estimator and showed that it has a better worse-case performance than the empirical mean. Mathieu (2022) \cite{mathieu2022concentration} further derived the concentration bounds of $M$-estimators and demonstrated that, by selecting the tuning parameter which depends on the variance, these $M$-estimator can also be a sub-Gaussian estimator. The concept of the median of means ($\text{MoM}_{k,b=\frac{n}{k},n}$) was first introduced by Nemirovsky and Yudin (1983) in their work on stochastic optimization \cite{nemirovskij1983problem}, while later was revisited in Jerrum, Valiant, and Vazirani (1986), \cite{jerrum1986random} and Alon, Matias and Szegedy (1996) \cite{alon1996space}'s works. Given its good performance even for distributions with infinite second moments, the MoM has received increasing attention over the past decade \cite{hsu2014heavy,devroye2016sub,laforgue2019medians,lecue2020robust}. Devroye, Lerasle, Lugosi, and Oliveira (2016) showed that $\text{MoM}_{k,b=\frac{n}{k},n}$ nears the optimum of sub-Gaussian mean estimation with regards to concentration bounds when the distribution has a heavy tail \cite{devroye2016sub}. Laforgue, Clemencon, and Bertail (2019) proposed the median of randomized means ($\text{MoRM}_{k,b,n}$) \cite{laforgue2019medians}, wherein, rather than partitioning, an arbitrary number, $b$, of blocks are built independently from the sample, and showed that $\text{MoRM}_{k,b,n}$ has a better non-asymptotic sub-Gaussian property compared to $\text{MoM}_{k,b=\frac{n}{k},n}$. In fact, asymptotically, the Hodges-Lehmann (H-L) estimator is equivalent to $\text{MoM}_{k=2,b=\frac{n}{k}}$ and $\text{MoRM}_{k=2,b}$, and they can be seen as the pairwise mean distribution is approximated by the sampling without replacement and bootstrap, respectively. When $k\ll n$, the difference between sampling with replacement and without replacement is negligible. For the asymptotic validity, readers are referred to the foundational works of Efron (1979) \cite{efron1979bootstrap}, Bickel and Freedman (1981, 1984) \cite{bickel1981some,bickel1984asymptotic}, and Helmers, Janssen, and Veraverbeke (1990) \cite{helmers1990bootstrapping}. 

Here, the $\epsilon$,$b$-stratified mean is defined as \begin{align*}\text{SM}_{\epsilon,b,n}\coloneqq\frac{b}{n}\left(\sum_{j=1}^{\frac{b-1}{2b\epsilon}}\sum_{i_j=\frac{\left(2bj-b-1\right)n\epsilon}{b-1}+1}^{\frac{\left(2bj-b+1\right)n\epsilon}{b-1}}X_{i_j}\right)\text{,}\end{align*} where $X_1\le \ldots \le X_n$ denote the order statistics of a sample of $n$ independent and identically distributed random variables $X_1$, $\ldots$, $X_n$. $b\in\mathbb{N}$, $b\geq3$, and $b\text{ mod }2=1$. The definition was further refined to guarantee the continuity of the breakdown point by incorporating an additional block in the center when $\lfloor\frac{b-1}{2b\epsilon}\rfloor\text{ mod }2=0$, or by adjusting the central block when $\lfloor\frac{b-1}{2b\epsilon}\rfloor\text{ mod }2=1$ (SI Text). If the subscript $n$ is omitted, only the asymptotic behavior is considered. If $b$ is omitted, $b=3$ is assumed. $\text{SM}_{\epsilon,b=3}$ is equivalent to $\text{STM}_{\epsilon}$, when $\epsilon>\frac{1}{6}$. When $\frac{b-1}{2\epsilon}\in\mathbb{N}$, the basic idea of the stratified mean is to distribute the data into $\frac{b-1}{2\epsilon}$ equal-sized non-overlapping blocks according to their order. Then, further sequentially group these blocks into $b$ equal-sized strata and compute the mean of the middle stratum, which is the median of means of each stratum. In situations where $i\text{ mod }1\neq0$, a potential solution is to generate multiple smaller samples that satisfy the equality by sampling without replacement, and subsequently calculate the mean of all estimations. The details of determining the smaller sample size and the number of sampling times are provided in the SI Text. Although the principle resembles that of the median of means, $\text{SM}_{\epsilon,b,n}$ is different from $\text{MoM}_{k=\frac{n}{b},b,n}$ as it does not include the random shift. Additionally, the stratified mean differs from the mean of the sample obtained through stratified sampling methods, introduced by Neyman (1934) \cite{neyman1934two} or ranked set sampling \cite{mcintyre1952method}, introduced by McIntyre in 1952, as these sampling methods aim to obtain more representative samples or improve the efficiency of sample estimates, but the sample means based on them are not robust. When $b\text{ mod }2=1$, the stratified mean can be regarded as replacing the other equal-sized strata with the middle stratum, which, in principle, is analogous to the Winsorized mean that replaces extreme values with less extreme percentiles. Furthermore, while the bounds confirm that the Winsorized mean and median of means outperform the trimmed mean \cite{bieniek2016comparison,danielak2003theory,devroye2016sub} in worst-case performance, the complexity of bound analysis makes it difficult to achieve a complete and intuitive understanding of these results. Also, a clear explanation for the average performance of them remains elusive. The aim of this paper is to define a series of semiparametric models using the signs of derivatives, reveal their elegant interrelations and connections to parametric models, and show that by exploiting these models, two sets of sophisticated mean estimators can be deduced, which exhibit strong robustness to departures from assumptions. 

\section*{Quantile Average and Weighted Average}\label{QA} The symmetric trimmed mean, symmetric Winsorized mean, and stratified mean are all $L$-estimators. More specifically, they are symmetric weighted averages, which are defined as \begin{align*}\text{SWA}_{\epsilon,n}\coloneqq\frac{\sum_{i=1}^{\lceil\frac{n}{2}\rceil}{\frac{X_i+X_{n-i+1}}{2}w_i}}{\sum_{i=1}^{\lceil\frac{n}{2}\rceil}{w_i}}\text{,}\end{align*}where $w_i$s are the weights applied to the symmetric quantile averages according to the definition of the corresponding $L$-estimators. For example, for the $\epsilon$-symmetric trimmed mean, $w_{i}=\left\{\begin{array}{@{}ll@{}}0, & i<n\epsilon \\1, & i\geq n\epsilon\end{array}\right.$, when $n\epsilon\in\mathbb{N}$. The mean and median are indeed two special cases of the symmetric trimmed mean. 

To extend the symmetric quantile average to the asymmetric case, two definitions for the $\epsilon$,$\gamma$-quantile average ($\text{QA}_{\epsilon,\gamma,n}$) are proposed. The first definition is: \begin{equation}\begin{split}\frac{1}{2}(\hat{Q}_{n}(\gamma\epsilon)+\hat{Q}_{n}(1-\epsilon))\text{,}\label{eq:1}\end{split}\end{equation}and the second definition is: \begin{equation}\begin{split}\frac{1}{2}(\hat{Q}_{n}(\epsilon)+\hat{Q}_{n}(1-\gamma\epsilon))\text{,}\label{eq:2}\end{split}\end{equation}\noindent where $\hat{Q}_{n}(p)$ is the empirical quantile function; $\gamma$ is used to adjust the degree of asymmetry, $\gamma\geq0$; and $0\leq\epsilon\leq\frac{1}{1+\gamma}$. For trimming from both sides, [\ref{eq:1}] and [\ref{eq:2}] are essentially equivalent. The first definition along with $\gamma\geq0$ and $0\leq\epsilon\leq\frac{1}{1+\gamma}$ are assumed in the rest of this article unless otherwise specified, since many common asymmetric distributions are right-skewed, and [\ref{eq:1}] allows trimming only from the right side by setting $\gamma=0$.

Analogously, the weighted average can be defined as \begin{align*}\text{WA}_{\epsilon,\gamma,n}\coloneqq\frac{\int_{0}^{\frac{1}{1+\gamma}}{\mathrm{QA}\left(\epsilon_0,\gamma,n\right)w(\epsilon_0) d\epsilon_0}}{\int_{0}^{\frac{1}{1+\gamma}}w(\epsilon_0) d\epsilon_0}\text{.}\end{align*}For any weighted average, if $\gamma$ is omitted, it is assumed to be 1. The $\epsilon,\gamma$-trimmed mean ($\text{TM}_{\epsilon,\gamma,n}$) is a weighted average with a left trim size of $n\gamma\epsilon$ and a right trim size of $n\epsilon$, where $w(\epsilon_0)=\left\{\begin{array}{@{}ll@{}}0, & \epsilon_0<\epsilon \\1, & \epsilon_0\geq \epsilon\end{array}\right.$. Using this definition, regardless of whether $n\gamma\epsilon \notin\mathbb{N}$ or $n\epsilon\notin\mathbb{N}$, the TM computation remains the same, since this definition is based on the empirical quantile function. However, in this article, considering the computational cost in practice, non-asymptotic definitions of various types of weighted averages are primarily based on order statistics. Unless stated otherwise, the solution to their decimal issue is the same as that in SM.

Furthermore, for weighted averages, separating the breakdown point into upper and lower parts is necessary.
\begin{definition}[Upper/lower breakdown point]\label{rbp}
The upper breakdown point is the breakdown point generalized in Davies and Gather (2005)'s paper \cite{Breakdown1}. The finite-sample upper breakdown point is the finite sample breakdown point defined by Donoho and Huber (1983) \cite{donoho1983notion} and also detailed in \cite{Breakdown1}. The (finite-sample) lower breakdown point is replacing the infinity symbol in these definitions with negative infinity.
\end{definition}

\section*{Classifying Distributions by the Signs of Derivatives}\label{A}Let $\mathcal{P}_{\mathbb{R}}$ denote the set of all continuous distributions over $\mathbb{R}$ and $\mathcal{P}_{\mathbb{X}}$ denote the set of all discrete distributions over a countable set $\mathbb{X}$. The default of this article will be on the class of continuous distributions, $\mathcal{P}_{\mathbb{R}}$. However, it's worth noting that most discussions and results can be extended to encompass the discrete case, $\mathcal{P}_{\mathbb{X}}$, unless explicitly specified otherwise. Besides fully and smoothly parameterizing them by a Euclidean parameter or merely assuming regularity conditions, there exist additional methods for classifying distributions based on their characteristics, such as their skewness, peakedness, modality, and supported interval. In 1956, Stein initiated the study of estimating parameters in the presence of an infinite-dimensional nuisance shape parameter \cite{stein1956efficient} and proposed a necessary condition for this type of problem, a contribution later explicitly recognized as initiating the field of semiparametric statistics \cite{bickel1982adaptive}. In 1982, Bickel simplified Stein's general heuristic necessary condition \cite{stein1956efficient}, derived sufficient conditions, and used them in formulating adaptive estimates \cite{bickel1982adaptive}. A notable example discussed in these groundbreaking works was the adaptive estimation of the center of symmetry for an unknown symmetric distribution, which is a semiparametric model. In 1993, Bickel, Klaassen, Ritov, and Wellner published an influential semiparametrics textbook \cite{bickel1993efficient}, which categorized most common statistical models as semiparametric models, considering parametric and nonparametric models as two special cases within this classification. Yet, there is another old and commonly encountered class of distributions that receives little attention in semiparametric literature: the unimodal distribution. It is a very unique semiparametric model because its definition is based on the signs of derivatives, i.e., ($f'(x)>0$ for $x\le M$) $\wedge$ ($f'(x)<0$ for $x\geq M$), where $f(x)$ is the probability density function (pdf) of a random variable $X$, $M$ is the mode. Let $\mathcal{P}_{U}$ denote the set of all unimodal distributions. There was a widespread misbelief that the median of an arbitrary unimodal distribution always lies between its mean and mode until Runnenburg (1978) and van Zwet (1979) \cite{runnenburg1978mean,zwet1979mean} endeavored to determine sufficient conditions for the mean-median-mode inequality to hold, thereby implying the possibility of its violation. The class of unimodal distributions that satisfy the mean-median-mode inequality constitutes a subclass of $\mathcal{P}_{U}$, denoted by $\mathcal{P}_{MMM}\subsetneq\mathcal{P}_{U}$. To further investigate the relations of location estimates within a distribution, the $\gamma$-orderliness for a right-skewed distribution is defined as \begin{align*}\forall 0\leq\epsilon_1\leq\epsilon_2\leq\frac{1}{1+\gamma}, \text{QA}(\epsilon_1,\gamma) \geq \text{QA}(\epsilon_2,\gamma).\end{align*} 

The necessary and sufficient condition below hints at the relation between the mean-median-mode inequality and the $\gamma$-orderliness. \begin{theorem}\label{q1}A distribution is $\gamma$-ordered if and only if its pdf satisfies the inequality $f(Q(\gamma\epsilon))\geq f(Q(1-\epsilon))$ for all $0\le \epsilon\le\frac{1}{1+\gamma}$ or $f(Q(\gamma\epsilon))\leq f(Q(1-\epsilon))$ for all $0\le \epsilon\le\frac{1}{1+\gamma}$.\end{theorem}\begin{proof}Without loss of generality, consider the case of right-skewed distribution. From the above definition of $\gamma$-orderliness, it is deduced that $\frac{Q(\gamma\epsilon-\delta)+Q(1-\epsilon+\delta)}{2}\geq\frac{Q(\gamma\epsilon)+Q(1-\epsilon)}{2}  \Leftrightarrow Q(\gamma\epsilon-\delta)-Q(\gamma\epsilon)\geq Q(1-\epsilon)-Q(1-\epsilon+\delta)  \Leftrightarrow Q'(1-\epsilon)\geq Q'(\gamma\epsilon)$, where $\delta$ is an infinitesimal positive quantity. Observing that the quantile function is the inverse function of the cumulative distribution function (cdf), $Q'(1-\epsilon)\geq Q'(\gamma\epsilon) \Leftrightarrow F'(Q(\gamma\epsilon))\geq F'(Q(1-\epsilon))$, thereby completing the proof, since the derivative of cdf is pdf.\end{proof} According to Theorem \ref{q1}, if a probability distribution is right-skewed and monotonic decreasing, it will always be $\gamma$-ordered. For a right-skewed unimodal distribution, if $Q(\gamma\epsilon)>M$, then the inequality $f(Q(\gamma\epsilon))\geq f(Q(1-\epsilon))$ holds. The principle is extendable to unimodal-like distributions. Suppose there is a right-skewed unimodal-like distribution with the first mode, denoted as $M_1$, having the greatest probability density, while there are several smaller modes located towards the higher values of the distribution. Furthermore, assume that this distribution follows the mean-$\gamma$-median-first mode inequality, and the $\gamma$-median, $Q(\frac{\gamma}{1+\gamma})$, falling within the first dominant mode (i.e., if $x>Q(\frac{\gamma}{1+\gamma})$, $f(Q(\frac{\gamma}{1+\gamma}))\geq f(x)$). Then, if $Q(\gamma\epsilon)>M_1$, the inequality $f(Q(\gamma\epsilon))\geq f(Q(1-\epsilon))$ also holds. In other words, even though a distribution following the mean-$\gamma$-median-mode inequality may not be strictly $\gamma$-ordered, the inequality defining the $\gamma$-orderliness remains valid for most quantile averages. The mean-$\gamma$-median-mode inequality can also indicate possible bounds for $\gamma$ in practice, e.g., for any distributions, when $\gamma\rightarrow\infty$, the $\gamma$-median will be greater than the mean and the mode, when $\gamma\rightarrow0$, the $\gamma$-median will be smaller than the mean and the mode, a reasonable $\gamma$ should maintain the validity of the mean-$\gamma$-median-mode inequality.

The definition above of $\gamma$-orderliness for a right-skewed distribution implies a monotonic decreasing behavior of the quantile average function with respect to the breakdown point. Therefore, consider the sign of the partial derivative, it can also be expressed as: \begin{align*}\forall 0\leq\epsilon\leq\frac{1}{1+\gamma}, \frac{\partial\text{QA}}{\partial\epsilon}\leq0.\end{align*}The left-skewed case can be obtained by reversing the inequality $\frac{\partial\text{QA}}{\partial\epsilon}\leq0$ to $\frac{\partial\text{QA}}{\partial\epsilon}\geq0$ and employing the second definition of QA, as given in [\ref{eq:2}]. For simplicity, the left-skewed case will be omitted in the following discussion. If $\gamma=1$, the $\gamma$-ordered distribution is referred to as ordered distribution. 

Furthermore, many common right-skewed distributions, such as the Weibull, gamma, lognormal, and Pareto distributions, are partially bounded, indicating a convex behavior of the QA function with respect to $\epsilon$ as $\epsilon$ approaches 0. By further assuming convexity, the second $\gamma$-orderliness can be defined for a right-skewed distribution as follows, \begin{align*}\forall 0\leq\epsilon\leq\frac{1}{1+\gamma}, \frac{\partial^{2}\text{QA}}{\partial\epsilon^{2}}\geq0  \wedge \frac{\partial\text{QA}}{\partial\epsilon}\leq0\text{.}\end{align*}Analogously, the $\nu$th $\gamma$-orderliness of a right-skewed distribution can be defined as $(-1)^{\nu}\frac{\partial^{\nu}\text{QA}}{\partial\epsilon^{\nu}}\geq0\wedge\ldots\wedge -\frac{\partial\text{QA}}{\partial\epsilon}\geq0$. If $\gamma=1$, the $\nu$th $\gamma$-orderliness is referred as to $\nu$th orderliness. Let $\mathcal{P}_{O}$ denote the set of all distributions that are ordered and $\mathcal{P}_{O_\nu}$ and $\mathcal{P}_{\gamma O_\nu}$ represent the sets of all distributions that are $\nu$th ordered and $\nu$th $\gamma$-ordered, respectively. When the shape parameter of the Weibull distribution, $\alpha$, is smaller than $\frac{1}{1-\ln (2)}$, it can be shown that the Weibull distribution belongs to $\mathcal{P}_{U}\cap\mathcal{P}_{O}\cap\mathcal{P}_{O_2}$ (SI Text). At $\alpha\approx3.602$, the Weibull distribution is symmetric, and as $\alpha\rightarrow\infty$, the skewness of the Weibull distribution approaches 1. Therefore, the parameters that prevent it from being included in the set correspond to cases when it is near-symmetric, as shown in the SI Text. Nevertheless, computing the derivatives of the QA function is often intricate and, at times, challenging. The following theorems establish the relationship between $\mathcal{P}_{O}$, $\mathcal{P}_{O_\nu}$, and $\mathcal{P}_{\gamma O_\nu}$, and a wide range of other semi-parametric distributions. They can be used to quickly identify some parametric distributions in $\mathcal{P}_{O}$, $\mathcal{P}_{O_\nu}$, and $\mathcal{P}_{\gamma O_\nu}$.

\begin{theorem}\label{lst}For any random variable $X$ whose probability distribution function belongs to a location-scale family, the distribution is $\nu$th $\gamma$-ordered if and only if the family of probability distributions is $\nu$th $\gamma$-ordered.\end{theorem}\begin{proof}Let $Q_0$ denote the quantile function of the standard distribution without any shifts or scaling. After a location-scale transformation, the quantile function becomes $Q(p)=\lambda Q_0(p)+\mu$, where $\lambda$ is the scale parameter and $\mu$ is the location parameter. According to the definition of the $\nu$th $\gamma$-orderliness, the signs of derivatives of the QA function are invariant after this transformation. As the location-scale transformation is reversible, the proof is complete.\end{proof}
Theorem \ref{lst} demonstrates that in the analytical proof of the $\nu$th $\gamma$-orderliness of a parametric distribution, both the location and scale parameters can be regarded as constants. It is also instrumental in proving other theorems. 

\begin{theorem}\label{gsdqi} Define a $\gamma$-symmetric distribution as one for which the quantile function satisfies $Q(\gamma \epsilon)=2Q(\frac{\gamma}{1+\gamma})-Q(1-\epsilon)$ for all $0\leq\epsilon\leq\frac{1}{1+\gamma}$. Any $\gamma$-symmetric distribution is $\nu$th $\gamma$-ordered. 
\end{theorem} \begin{proof} The equality, $Q(\gamma \epsilon)=2Q(\frac{\gamma}{1+\gamma})-Q(1-\epsilon)$, implies that $\frac{\partial Q(\gamma \epsilon)}{\partial\epsilon}=\gamma Q'(\gamma\epsilon)=\frac{\partial(-Q(1-\epsilon))}{\partial\epsilon}=Q'(1-\epsilon)$. From the first definition of QA, the QA function of the $\gamma$-symmetric distribution is a horizontal line, since $\frac{\partial\text{QA}}{\partial\epsilon}=\gamma Q'(\gamma\epsilon)-Q'(1-\epsilon)=0$. So, the $\nu$th order derivative of QA is always zero. 
\end{proof}
\begin{theorem}\label{sdqi}A symmetric distribution is a special case of the $\gamma$-symmetric distribution when $\gamma=1$, provided that the cdf is monotonic. \end{theorem} \begin{proof} A symmetric distribution is a probability distribution such that for all $x$, $f(x)=f(2m- x)$. Its cdf satisfies $F(x) = 1 - F(2m- x)$. Let $x=Q(p)$, then, $F(Q(p))=p= 1 - F(2m- Q(p))$ and $F(Q(1-p)) = 1 - p\Leftrightarrow p=1-F(Q(1-p))$. Therefore, $F(2m- Q(p))=F(Q(1-p))$. Since the cdf is monotonic, $2m- Q(p)=Q(1-p)\Leftrightarrow Q(p)=2m-Q(1-p)$. Choosing $p=\epsilon$ yields the desired result. \end{proof} Since the generalized Gaussian distribution is symmetric around the median, it is $\nu$th ordered, as a consequence of Theorem \ref{gsdqi}. Also, the integral of all quantile averages is not equal to the mean, unless $\gamma=1$, as the left and right parts have different weights. The symmetric distribution has a unique role in that its all quantile averages are equal to the mean for a distribution with a finite mean.

\begin{theorem}\label{kqin}Any right-skewed distribution whose quantile function $Q$ satisfies $Q^{(\nu)}\left(p\right)\geq0\wedge\ldots Q^{(i)}\left(p\right)\geq0\ldots\wedge Q^{(2)}\left(p\right)\geq0$, $i\text{ mod }2=0$, is $\nu$th $\gamma$-ordered, provided that $0\leq\gamma\leq1$.\end{theorem}\begin{proof}Since $(-1)^{i}\frac{\partial^{i}\text{QA}}{\partial\epsilon^{i}}=\frac{1}{2}((-\gamma)^{i}Q^{i}(\gamma\epsilon)+Q^{i}(1-\epsilon))$ and $1\leq i \leq\nu$, when $i\text{ mod }2=0$, $(-1)^{i}\frac{\partial^{i}\text{QA}}{\partial\epsilon^{i}}\geq0$ for all $\gamma\geq0$. When $i\text{ mod }2=1$, if further assuming $0\leq\gamma\leq1$, $(-1)^{i}\frac{\partial^{i}\text{QA}}{\partial\epsilon^{i}}\geq0$, since $Q^{(i+1)}\left(p\right)\geq0$.\end{proof}This result makes it straightforward to show that the Pareto distribution follows the $\nu$th $\gamma$-orderliness, provided that $0\leq\gamma\leq1$, since the quantile function of the Pareto distribution is $Q_{Par}\left(p\right)=x_m(1-p)^{-\frac{1}{\alpha}}$, where $x_m>0$, $\alpha>0$, and so $Q_{Par}^{(\nu)}\left(p\right)\geq0$ for all $\nu\in\mathbb{N}$ according to the chain rule.

\begin{theorem}\label{mpdfti}A right-skewed distribution with a monotonic decreasing pdf is second $\gamma$-ordered.\end{theorem}\begin{proof}Given that a monotonic decreasing pdf implies $f'(x)=F^{\left(2\right)}\left(x\right)\le0$, let $x=Q\left(F\left(x\right)\right)$, then by differentiating both sides of the equation twice, one can obtain  $0=Q^{\left(2\right)}\left(F\left(x\right)\right)\left(F^\prime\left(x\right)\right)^2+Q^\prime\left(F\left(x\right)\right)F^{(2)}\left(x\right)\Rightarrow Q^{\left(2\right)}\left(F\left(x\right)\right)=-\frac{Q^\prime\left(F\left(x\right)\right)F^{\left(2\right)}\left(x\right)}{\left(F^\prime\left(x\right)\right)^2}\geq0$, since $Q^\prime\left(p\right)\geq0$. Theorem \ref{q1} already established the $\gamma$-orderliness for all $\gamma\geq0$, which means $\forall 0\leq\epsilon\leq\frac{1}{1+\gamma}, \frac{\partial\text{QA}}{\partial\epsilon}\leq0.$ The desired result is then derived from the proof of Theorem \ref{kqin}, since $(-1)^{2}\frac{\partial^{2}\text{QA}}{\partial\epsilon^{2}}\geq0$ for all $\gamma\geq0$.\end{proof}

Theorem \ref{mpdfti} provides valuable insights into the relation between modality and second $\gamma$-orderliness. The conventional definition states that a distribution with a monotonic pdf is still considered unimodal. However, within its supported interval, the mode number is zero. Theorem \ref{q1} implies that the number of modes and their magnitudes within a distribution are closely related to the likelihood of $\gamma$-orderliness being valid. This is because, for a distribution satisfying the necessary and sufficient condition in Theorem \ref{q1}, it is already implied that the probability density of the left-hand side of the $\gamma$-median is always greater than the corresponding probability density of the right-hand side of the $\gamma$-median. So although counterexamples can always be constructed for non-monotonic distributions, the general shape of a $\gamma$-ordered distribution should have a single dominant mode. It can be easily established that the gamma distribution is second $\gamma$-ordered when $\alpha\leq1$, as the pdf of the gamma distribution is $f\left(x\right)=\frac{\lambda ^{-\alpha} x^{\alpha-1} e^{-\frac{x}{\lambda }}}{\Gamma (\alpha)}$, where $x\geq0$, $\lambda>0$, $\alpha>0$, and $\Gamma$ represents the gamma function. This pdf is a product of two monotonic decreasing functions under constraints. For $\alpha>1$, analytical analysis becomes challenging. Numerical results can varify that orderliness is valid if $\alpha<140$, the second orderliness is valid if $\alpha>81$, and the third orderliness is valid if $\alpha<59$ (SI Text). It is instructive to consider that when $\alpha\rightarrow\infty$, the gamma distribution converges to a Gaussian distribution with mean $\mu=\alpha\lambda$ and variance $\sigma=\alpha\lambda^2$. The skewness of the gamma distribution, $\frac{\alpha+2}{\sqrt{\alpha (\alpha+1)}}$, is monotonic with respect to $\alpha$, since $\frac{\partial\Tilde{\mu}_3(\alpha)}{\partial\alpha}=\frac{-3 \alpha-2}{2 (\alpha (\alpha+1))^{3/2}}<0$. When $\alpha=59$, $\Tilde{\mu}_3(\alpha)=1.025$. Theorefore, similar to the Weibull distribution, the parameters which make these distributions fail to be included in $\mathcal{P}_{U}\cap\mathcal{P}_{O}\cap\mathcal{P}_{O_2}\cap\mathcal{P}_{O_3}$ also correspond to cases when it is near-symmetric. 

\begin{theorem}\label{mtti}Consider a $\gamma$-symmetric random variable $X$. Let it be transformed using a function $\phi\left(x\right)$ such that $\phi^{(2)}\left(x\right)\geq0$ over the interval supported, the resulting convex transformed distribution is $\gamma$-ordered. Moreover, if the quantile function of $X$ satifies $Q^{(2)}\left(p\right)\le0$, the convex transformed distribution is second $\gamma$-ordered.\end{theorem}\begin{proof}Let $\phi\mathrm{QA}(\epsilon,\gamma)=\frac{1}{2}(\phi (Q(\gamma\epsilon))+\phi (Q(1-\epsilon)))$. Then, for all $0\leq\epsilon\leq\frac{1}{1+\gamma}$, $\frac{\partial\phi\mathrm{QA}}{\partial\epsilon}=\frac{1}{2}\left(\gamma\phi^\prime\left(Q\left(\gamma\epsilon\right)\right)Q^\prime\left(\gamma\epsilon\right)-\phi^\prime\left(Q\left(1-\epsilon\right)\right)Q^\prime\left(1-\epsilon\right)\right)=\frac{1}{2} \gamma Q^\prime\left(\gamma\epsilon\right)\left(\phi^\prime\left(Q\left(\gamma\epsilon\right)\right)-\phi^\prime\left(Q\left(1-\epsilon\right)\right)\right)\le0$, since for a $\gamma$-symmetric distribution, $Q(\frac{1}{1+\gamma})-Q\left(\gamma\epsilon\right)=Q\left(1-\epsilon\right)-Q(\frac{1}{1+\gamma})$, differentiating both sides, $-\gamma Q^\prime\left(\gamma\epsilon\right)=-Q^\prime(1-\epsilon)$, where $Q^\prime\left(p\right)\geq0, \phi^{(2)}\left(x\right)\geq0$. If further differentiating the equality, $\gamma^2 Q^{(2)}\left(\gamma\epsilon\right)=-Q^{(2)}(1-\epsilon)$. Since $\frac{\partial^{(2)}\phi\mathrm{QA} }{\partial\epsilon^{(2)}}=\frac{1}{2}\left(\gamma^2 \phi^2\left(Q\left(\gamma\epsilon\right)\right)\left(Q^\prime\left(\gamma\epsilon\right)\right)^2+\phi^2\left(Q\left(1-\epsilon\right)\right)\left(Q^\prime\left(1-\epsilon\right)\right)^2\right)+\frac{1}{2}\left(\gamma^2\phi^\prime\left(Q\left(\gamma\epsilon\right)\right)\left(Q^2\left(\gamma\epsilon\right)\right)+\phi^\prime\left(Q\left(1-\epsilon\right)\right)\left(Q^2\left(1-\epsilon\right)\right)\right)=\frac{1}{2}\left(\left(\phi^{\left(2\right)}\left(Q\left(\gamma\epsilon\right)\right)+\phi^{\left(2\right)}\left(Q\left(1-\epsilon\right)\right)\right)\left(\gamma^2 Q^\prime\left(\gamma\epsilon\right)\right)^2\right)+\frac{1}{2}\left(\left(\phi^\prime\left(Q\left(\gamma\epsilon\right)\right)-\phi^\prime\left(Q\left(1-\epsilon\right)\right)\right)\gamma^2 Q^{(2)}\left(\gamma\epsilon\right)\right)$. If $Q^{(2)}\left(p\right)\leq0$, for all $0\leq\epsilon\leq\frac{1}{1+\gamma}$, $\frac{\partial^{(2)}\phi\mathrm{QA} }{\partial\epsilon^{(2)}}\geq0$.\end{proof}


An application of Theorem \ref{mtti} is that the lognormal distribution is ordered as it is exponentially transformed from the Gaussian distribution. The quantile function of the Gaussian distribution meets the condition $Q_{Gau}^{(2)}\left(p\right)=-2 \sqrt{2} \pi  \sigma e^{2 \text{erfc}^{-1}(2 p)^2} \text{erfc}^{-1}(2 p)\le0$, where $\sigma$ is the standard deviation of the Gaussian distribution and erfc denotes the complementary error function. Thus, the lognormal distribution is second ordered. Numerical results suggest that it is also third ordered, although analytically proving this result is challenging.

Theorem \ref{mtti} also reveals a relation between convex transformation and orderliness, since $\phi$ is the non-decreasing convex function in van Zwet's trailblazing work \emph{Convex transformations of random variables} \cite{van1964convex} if adding an additional constraint that $\phi^\prime\left(x\right)\geq0$. Consider a near-symmetric distribution $S$, such that the $\text{SQA}(\epsilon)$ as a function of $\epsilon$ fluctuates from $0$ to $\frac{1}{2}$. By definition, $S$ is not ordered. Let $s$ be the pdf of $S$. Applying the transformation $\phi\left(x\right)$ to $S$ decreases $s(Q_S(\epsilon))$, and the decrease rate, due to the order, is much smaller for $s(Q_S(1-\epsilon))$. As a consequence, as $\phi^{(2)}\left(x\right)$ increases, eventually, after a point, for all $0\leq\epsilon\leq\frac{1}{1+\gamma}$, $s(Q_S(\epsilon))$ becomes greater than $s(Q_S(1-\epsilon))$ even if it was not previously. Thus, the $\text{SQA}({\epsilon})$ function becomes monotonically decreasing, and $S$ becomes ordered. Accordingly, in a family of distributions that differ by a skewness-increasing transformation in van Zwet's sense, violations of orderliness typically occur only when the distribution is near-symmetric.

Pearson proposed using the 3 times standardized mean-median difference, $\frac{3(\mu-m)}{\sigma}$, as a measure of skewness in 1895 \cite{pearson1895x}. Bowley (1926) proposed a measure of skewness based on the $\text{SQA}_{\epsilon=\frac{1}{4}}$-median difference $\text{SQA}_{\epsilon=\frac{1}{4}}-m$ \cite{bowley1926elements}. Groeneveld and Meeden (1984) \cite{groeneveld1984measuring} generalized these measures of skewness based on van Zwet's convex transformation \cite{van1964convex} while exploring their properties. A distribution is called monotonically right-skewed if and only if $\forall 0\leq\epsilon_1\leq\epsilon_2\leq\frac{1}{2}, \text{SQA}_{\epsilon_1}-m \geq \text{SQA}_{\epsilon_2}-m.$ Since $m$ is a constant, the monotonic skewness is equivalent to the orderliness. For a nonordered distribution, the signs of $\text{SQA}_{\epsilon}-m$ with different breakdown points might be different, implying that some skewness measures indicate left-skewed distribution, while others suggest right-skewed distribution. Although it seems reasonable that such a distribution is likely be generally near-symmetric, counterexamples can be constructed. For example, first consider the Weibull distribution, when $\alpha>\frac{1}{1-\ln (2)}$, it is near-symmetric and nonordered, the non-monotonicity of the SQA function arises when $\epsilon$ is close to $\frac{1}{2}$, but if then replacing the third quartile with one from a right-skewed heavy-tailed distribution leads to a right-skewed, heavy-tailed, and nonordered distribution. Therefore, the validity of robust measures of skewness based on the SQA-median difference is closely related to the orderliness of the distribution. 

Remarkably, in 2018, Li, Shao, Wang, Yang \cite{li2018worst} proved the bias bound of any quantile for arbitrary continuous distributions with finite second moments. Here, let $\mathcal{P}_{\mu,\sigma}$ denotes the set of continuous distributions whose mean is $\mu$ and standard deviation is $\sigma$. The bias upper bound of the quantile average for $P\in\mathcal{P}_{\mu=0,\sigma=1}$ is given in the following theorem.

\begin{theorem}\label{qabbb0}
The bias upper bound of the quantile average for any continuous distribution whose mean is zero and standard deviation is one is 
\begin{align*}
  \smashoperator[r]{\sup_{P\in \mathcal{P}_{\mu=0,\sigma=1}}}{\text{QA}}(\epsilon,\gamma)= \frac{1}{2}\left(\sqrt{\frac{\gamma\epsilon}{1-\gamma\epsilon}}+\sqrt{\frac{1-\epsilon}{\epsilon}}\right),
\end{align*}
where $0\leq\epsilon\leq\frac{1}{1+\gamma}$.
\end{theorem}
\begin{proof}
Since $\sup_{P\in \mathcal{P}_{\mu=0,\sigma=1}}{\frac{1}{2}(Q(\gamma\epsilon)+Q(1-\epsilon))}\leq\frac{1}{2}(\sup_{P\in \mathcal{P}_{\mu=0,\sigma=1}}{Q(\gamma\epsilon)}+\sup_{P\in \mathcal{P}_{\mu=0,\sigma=1}}{Q(1-\epsilon))}$, the assertion follows directly from the Lemma 2.6 in \cite{li2018worst}.
\end{proof}

In 2020, Bernard et al. \cite{bernard2020range} further refined these bounds for unimodal distributions and derived the bias bound of the symmetric quantile average. Here, the bias upper bound of the quantile average, $0\leq\gamma<5$, for $P\in \mathcal{P}_{U}\cap\mathcal{P}_{\mu=0,\sigma=1}$ is given as \begin{align*}\smashoperator[r]{\sup_{P\in \mathcal{P}_{U}\cap\mathcal{P}_{\mu=0,\sigma=1}}}{\text{QA}}(\epsilon,\gamma)= \begin{cases} \frac{1}{2}\left(\sqrt{\frac{4}{9\epsilon}-1}+\sqrt{\frac{3\gamma\epsilon}{4-3\gamma\epsilon}}\right) & 0\leq\epsilon\leq\frac{1}{6} \\ \frac{1}{2}\left(\sqrt{\frac{3(1-\epsilon)}{4-3(1-\epsilon)}}+\sqrt{\frac{3\gamma\epsilon}{4-3\gamma\epsilon}}\right) & \frac{1}{6}<\epsilon\leq\frac{1}{1+\gamma}. \end{cases} \end{align*} The proof based on the bias bounds of any quantile \cite{bernard2020range} and the $\gamma\geq5$ case are given in the SI Text. Subsequent theorems reveal the safeguarding role these bounds play in defining estimators based on $\nu$th $\gamma$-orderliness. 

\begin{theorem}\label{qabbb}$\sup_{P\in \mathcal{P}_{\mu=0,\sigma=1}}{\text{QA}}(\epsilon,\gamma)$ is monotonic decreasing with respect to $\epsilon$ over $[0,\frac{1}{1+\gamma}]$, provided that $0\leq\gamma\leq1$.\end{theorem}

\begin{proof} $\frac{\partial \sup{\text{QA}}(\epsilon,\gamma)}{\partial\epsilon}=\frac{1}{4} \left(\frac{\gamma }{\sqrt{\frac{\gamma  \epsilon }{1-\gamma  \epsilon }} (\gamma  \epsilon -1)^2}-\frac{1}{\sqrt{\frac{1}{\epsilon }-1} \epsilon ^2}\right)$. When $\gamma=0$, $\frac{\partial \sup{\text{QA}}(\epsilon,\gamma)}{\partial\epsilon}=\frac{1}{4} \left(\frac{\sqrt{\gamma} }{\sqrt{\frac{ \epsilon }{1-\gamma  \epsilon }} (\gamma  \epsilon -1)^2}-\frac{1}{\sqrt{\frac{1}{\epsilon }-1} \epsilon ^2}\right)=-\frac{1}{\sqrt{\frac{1}{\epsilon }-1} \epsilon ^2}\leq0$. When $\epsilon\rightarrow0^{+}$, $\lim_{\epsilon \to0^{+}}{\left(\frac{1}{4} \left(\frac{\gamma }{\sqrt{\frac{\gamma  \epsilon }{1-\gamma  \epsilon }} (\gamma  \epsilon -1)^2}-\frac{1}{\sqrt{\frac{1}{\epsilon }-1} \epsilon ^2}\right)\right)}=\lim_{\epsilon\rightarrow0^+}{\left(\frac{1}{4}\left(\frac{\sqrt\gamma}{\sqrt\epsilon}-\frac{1}{\sqrt{\epsilon^3}}\right)\right)}\to-\infty$. Assuming $\epsilon>0$, when $0<\gamma\leq1$, to prove $\frac{\partial \sup{\text{QA}}(\epsilon,\gamma)}{\partial\epsilon}\leq0$, it is equivalent to showing $\frac{\sqrt{\frac{\gamma\epsilon}{1-\gamma\epsilon}}\left(\gamma\epsilon-1\right)^2}{\gamma}\geq\sqrt{\frac{1}{\epsilon}-1}\epsilon^2$. Define $L(\epsilon,\gamma)=\frac{\sqrt{\frac{\gamma\epsilon}{1-\gamma\epsilon}}\left(\gamma\epsilon-1\right)^2}{\gamma}$, $R(\epsilon,\gamma)=\sqrt{\frac{1}{\epsilon}-1}\epsilon^2$. $\frac{L(\epsilon,\gamma)}{\epsilon^2}=\frac{\sqrt{\frac{\gamma\epsilon}{1-\gamma\epsilon}}\left(\gamma\epsilon-1\right)^2}{\gamma\epsilon^2}=\frac{1}{\gamma}\sqrt{\frac{1}{\frac{1}{\gamma\epsilon}-1}}\left(\gamma-\frac{1}{\epsilon}\right)^2$, $\frac{R(\epsilon,\gamma)}{\epsilon^2}=\sqrt{\frac{1}{\epsilon}-1}$. Then, $\frac{L(\epsilon,\gamma)}{\epsilon^2}\geq\frac{R(\epsilon,\gamma)}{\epsilon^2}\Leftrightarrow\frac{1}{\gamma}\sqrt{\frac{1}{\frac{1}{\gamma\epsilon}-1}}\left(\gamma-\frac{1}{\epsilon}\right)^2\geq\sqrt{\frac{1}{\epsilon}-1}\Leftrightarrow\frac{1}{\gamma}\left(\gamma-\frac{1}{\epsilon}\right)^2\geq\sqrt{\frac{1}{\epsilon}-1}\sqrt{\frac{1}{\gamma\epsilon}-1}$. Let $LmR\left(\frac{1}{\epsilon}\right)=\frac{1}{\gamma^2}\left(\gamma-\frac{1}{\epsilon}\right)^4-\left(\frac{1}{\epsilon}-1\right)\left(\frac{1}{\gamma\epsilon}-1\right)$. $\frac{\partial LmR\left(1/\epsilon\right)}{\partial\left(1/\epsilon\right)}=-\frac{4 (\gamma -\frac{1}{\epsilon})^3}{\gamma ^2}-\frac{\frac{1}{\epsilon}-1}{\gamma }-\frac{1}{\gamma\epsilon}+1=\frac{-4 \gamma ^3+\gamma ^2+\gamma +4 \frac{1}{\epsilon^3}-12 \gamma  \frac{1}{\epsilon^2}+12 \gamma ^2 \frac{1}{\epsilon}-2 \gamma  \frac{1}{\epsilon}}{\gamma ^2}$. Since $0\leq\gamma\leq1$, $0\le\epsilon\le\frac{1}{1+\gamma}\Leftrightarrow 0\leq\gamma\le\frac{1}{\epsilon}-1\Leftrightarrow 1-\frac{1}{\epsilon}\leq-\gamma\leq0 \Leftrightarrow   1\le\frac{1}{\epsilon}-\gamma\le\frac{1}{\epsilon}$. The numerator of $\frac{\partial LmR\left(1/\epsilon\right)}{\partial\left(1/\epsilon\right)}$ can be simplified as $-4\gamma^3+\gamma^2+\gamma+4\frac{1}{\epsilon^3}-12\gamma\frac{1}{\epsilon^2}+12\gamma^2\frac{1}{\epsilon}-2\gamma\frac{1}{\epsilon}=4\left(\frac{1}{\epsilon}-\gamma\right)^3+\gamma^2+\gamma-2\gamma\frac{1}{\epsilon}=4\left(\frac{1}{\epsilon}-\gamma\right)^3-\gamma^2+\gamma-2\gamma\left(\frac{1}{\epsilon}-\gamma\right)=\gamma\left(1-\gamma\right)+2\left(\frac{1}{\epsilon}-\gamma\right)\left(2\left(\frac{1}{\epsilon}-\gamma\right)^2-\gamma\right)$. Since $2\left(\frac{1}{\epsilon}-\gamma\right)^2\geq2$, $2\left(\frac{1}{\epsilon}-\gamma\right)^2-\gamma\geq2$. Also, $\gamma\left(1-\gamma\right)\geq0$, $\left(\frac{1}{\epsilon}-\gamma\right)\geq0$, therefore, $\gamma\left(1-\gamma\right)+2\left(\frac{1}{\epsilon}-\gamma\right)\left(2\left(\frac{1}{\epsilon}-\gamma\right)^2-\gamma\right)\geq0$, $\frac{\partial LmR\left(1/\epsilon\right)}{\partial\left(1/\epsilon\right)}\geq0$. Also, $LmR\left(1+\gamma\right)=\frac{1}{\gamma^2}\left(\gamma-1-\gamma\right)^4-\left(1+\gamma-1\right)\left(\frac{1}{\gamma}\left(1+\gamma\right)-1\right)=\frac{1}{\gamma^2}\geq0$. Therefore, $LmR\left(\frac{1}{\epsilon}\right)\geq0$ for $\epsilon\in(0,\frac{1}{1+\gamma}]$, provided that $0<\gamma\leq1$. Consequently, the simplified inequality $\frac{1}{\gamma}\left(\gamma-\frac{1}{\epsilon}\right)^2\geq\sqrt{\frac{1}{\epsilon}-1}\sqrt{\frac{1}{\gamma\epsilon}-1}$ is valid. $\frac{\partial \sup{\text{QA}}(\epsilon,\gamma)}{\partial\epsilon}$ is non-positive throughout the interval $0\leq\epsilon\leq\frac{1}{1+\gamma}$, given that $0\leq\gamma\leq1$, the proof is complete. \end{proof}

\begin{theorem}\label{sqabb} $\sup_{P\in \mathcal{P}_{U}\cap\mathcal{P}_{\mu=0,\sigma=1}}{\text{QA}}(\epsilon,\gamma)$ is a nonincreasing function with respect to $\epsilon$ on the interval $[0,\frac{1}{1+\gamma}]$, provided that $0\leq\gamma\leq1$.\end{theorem}
\begin{proof}
When $0\leq\epsilon\leq\frac{1}{6}$, $\frac{\partial \sup{\text{QA}}}{\partial\epsilon}=\frac{\gamma}{\sqrt{\frac{\epsilon \gamma}{12-9 \epsilon \gamma}} (4-3 \epsilon \gamma)^2}-\frac{1}{3 \sqrt{\frac{4}{\epsilon}-9} \epsilon^2}=\frac{\sqrt\gamma}{\sqrt{\frac{\epsilon}{12-9\epsilon\gamma}}\left(4-3\epsilon\gamma\right)^2}-\frac{1}{3 \sqrt{\frac{4}{\epsilon}-9} \epsilon^2}$. If $\gamma=0$ and $\epsilon\rightarrow0^{+}$, $\frac{\partial \sup{\text{QA}}}{\partial\epsilon}=-\frac{1}{3 \sqrt{\frac{4}{\epsilon}-9} \epsilon^2}<0$. If $\epsilon\rightarrow0^{+}$, $\lim_{\epsilon \to0^{+}}{\left(\frac{\gamma }{(4-3 \gamma  \epsilon )^2 \sqrt{\frac{\epsilon \gamma }{12-9 \gamma  \epsilon }}}-\frac{1}{3 \sqrt{\frac{4}{\epsilon }-9} \epsilon ^2}\right)}=\lim_{\epsilon \to0^{+}}{\left(\frac{\sqrt{3\gamma} }{\sqrt{4^3 \epsilon }}-\frac{1}{6 \sqrt{\epsilon ^3 }}\right)}\to-\infty$, for all $0\leq\gamma\leq1$, so, $\frac{\partial \sup{\text{QA}}}{\partial\epsilon}<0$. When $0<\epsilon\leq\frac{1}{6}$ and $0<\gamma\leq1$, to prove $\frac{\partial \sup{\text{QA}}}{\partial\epsilon}\leq0$, it is equivalent to showing $\frac{\sqrt{\frac{\epsilon \gamma}{12-9 \epsilon \gamma}} (4-3 \epsilon \gamma)^2}{\gamma}\geq3 \sqrt{\frac{4}{\epsilon}-9} \epsilon^2$. Define $L(\epsilon,\gamma)=\frac{\sqrt{\frac{\epsilon \gamma}{12-9 \epsilon \gamma}} (4-3 \epsilon \gamma)^2}{\gamma}$, $R(\epsilon,\gamma)=3\sqrt{\frac{4}{\epsilon}-9}\epsilon^2$. $\frac{L(\epsilon,\gamma)}{\epsilon^2}=\frac{\sqrt{\frac{\epsilon \gamma}{12-9 \epsilon \gamma}} (4-3 \epsilon \gamma)^2}{\gamma\epsilon^2}=\frac{1}{\gamma}\left(\frac{4}{\epsilon}-3\gamma\right)^2\sqrt{\frac{1}{\frac{12}{\epsilon\gamma}-9}}$, $\frac{R(\epsilon,\gamma)}{\epsilon^2}=3\sqrt{\frac{4}{\epsilon}-9}$. Then, the above inequality is equivalent to $\frac{L(\epsilon,\gamma)}{\epsilon^2}\geq\frac{R(\epsilon,\gamma)}{\epsilon^2}\Leftrightarrow\frac{1}{\gamma}\sqrt{\frac{1}{\frac{12}{\epsilon\gamma}-9}}\left(\frac{4}{\epsilon}-3\gamma\right)^2\geq3\sqrt{\frac{4}{\epsilon}-9}\Leftrightarrow\frac{1}{\gamma}\left(\frac{4}{\epsilon}-3\gamma\right)^2\geq3\sqrt{\frac{12}{\epsilon\gamma}-9}\sqrt{\frac{4}{\epsilon}-9}\Leftrightarrow \frac{1}{\gamma^2}\left(\frac{4}{\epsilon}-3\gamma\right)^4\geq9\left(\frac{12}{\epsilon\gamma}-9\right)\left(\frac{4}{\epsilon}-9\right)$. Let $LmR\left(\frac{1}{\epsilon}\right)=\frac{1}{\gamma^2}\left(\frac{4}{\epsilon}-3\gamma\right)^4-9\left(\frac{12}{\epsilon\gamma}-9\right)\left(\frac{4}{\epsilon}-9\right)$. $\frac{\partial LmR\left(1/\epsilon\right)}{\partial\left(1/\epsilon\right)}=\frac{16\left(\frac{4}{\epsilon}-3\gamma\right)^3}{\gamma^2}-36\left(\frac{12}{\epsilon\gamma}-9\right)-\frac{108\left(4\frac{4}{\epsilon}-9\right)}{\gamma}=\frac{4 \left(4 (\frac{4}{\epsilon}-3 \gamma )^3-27 \gamma  (\frac{4}{\epsilon}-3 \gamma )+27 (9-\frac{4}{\epsilon}) \gamma \right)}{\gamma ^2}=\frac{4 \left(256 \frac{1}{\epsilon}^3-576 \frac{1}{\epsilon}^2 \gamma +432 \frac{1}{\epsilon} \gamma ^2-216 \frac{1}{\epsilon} \gamma -108 \gamma ^3+81 \gamma ^2+243 \gamma \right)}{\gamma ^2}$. Since $256 \frac{1}{\epsilon}^3-576 \frac{1}{\epsilon}^2 \gamma +432 \frac{1}{\epsilon} \gamma ^2-216 \frac{1}{\epsilon} \gamma -108 \gamma ^3+81 \gamma ^2+243 \gamma\geq1536 \frac{1}{\epsilon}^2-576 \frac{1}{\epsilon}^2 +432 \frac{1}{\epsilon} \gamma ^2-216 \frac{1}{\epsilon} \gamma -108 \gamma ^3+81 \gamma ^2+243 \gamma\geq924 \frac{1}{\epsilon}^2+36 \frac{1}{\epsilon}^2-216 \frac{1}{\epsilon}+432 \frac{1}{\epsilon} \gamma ^2 -108 \gamma ^3+81 \gamma ^2+243 \gamma\geq924 \frac{1}{\epsilon}^2+36 \frac{1}{\epsilon}^2-216 \frac{1}{\epsilon}+513 \gamma ^2 -108 \gamma ^3+243 \gamma>0$, $\frac{\partial LmR\left(1/\epsilon\right)}{\partial\left(1/\epsilon\right)}>0$. Also, $LmR\left(6\right)=\frac{81 (\gamma -8) \left((\gamma -8)^3+15 \gamma \right)}{\gamma ^2}>0\Longleftrightarrow\gamma ^4-32 \gamma ^3+399 \gamma ^2-2168 \gamma +4096>0$. If $0<\gamma\leq1$, then $32 \gamma ^3<256$. Also, $\gamma^4>0$. So, it suffices to prove that $399 \gamma ^2-2168 \gamma +4096>256$. Applying the quadratic formula demonstrates the validity of $LmR\left(6\right)>0$, if $0<\gamma\leq1$. Hence, $LmR\left(\frac{1}{\epsilon}\right)\geq0$ for $\epsilon\in(0,\frac{1}{6}]$, if $0<\gamma\leq1$. The first part is finished.

When $\frac{1}{6}<\epsilon\leq\frac{1}{1+\gamma}$, $\frac{\partial \sup{\text{QA}}}{\partial\epsilon}=\sqrt3\left(\frac{\gamma}{\sqrt{\gamma\epsilon}\left(4-3\gamma\epsilon\right)^\frac{3}{2}}-\frac{1}{\sqrt{1-\epsilon}\left(3\epsilon+1\right)^\frac{3}{2}}\right)$. If $\gamma=0$,  $\frac{\gamma}{\sqrt{\gamma\epsilon}\left(4-3\gamma\epsilon\right)^\frac{3}{2}}=\frac{\sqrt{\gamma}}{\sqrt{\epsilon}\left(4-3\gamma\epsilon\right)^\frac{3}{2}}=0$, so $\frac{\partial \sup{\text{QA}}}{\partial\epsilon}=\sqrt3\left(-\frac{1}{\sqrt{1-\epsilon}\left(3\epsilon+1\right)^\frac{3}{2}}\right)<0$, for all $\frac{1}{6}<\epsilon\leq\frac{1}{1+\gamma}$. If $\gamma>0$, to determine whether $\frac{\partial \sup{\text{QA}}}{\partial\epsilon}\leq0$, when $\frac{1}{6}<\epsilon\leq\frac{1}{1+\gamma}$, since $\sqrt{1-\epsilon}\left(3\epsilon+1\right)^\frac{3}{2}>0$ and $\sqrt{\gamma\epsilon}\left(4-3\gamma\epsilon\right)^\frac{3}{2}>0$, showing $\frac{\sqrt{\gamma\epsilon}\left(4-3\gamma\epsilon\right)^\frac{3}{2}}{\gamma}\geq\sqrt{1-\epsilon}\left(3\epsilon+1\right)^\frac{3}{2}\Leftrightarrow\frac{\gamma\epsilon\left(4-3\gamma\epsilon\right)^3}{\gamma^2}\geq(1-\epsilon)\left(3\epsilon+1\right)^3\Leftrightarrow-27 \gamma ^2 \epsilon ^4+108 \gamma  \epsilon ^3+\frac{64 \epsilon }{\gamma }+27 \epsilon ^4-162 \epsilon ^2-8 \epsilon -1\geq0$ is sufficient. When $0<\gamma\leq1$, the inequality can be further simplified to $108 \gamma  \epsilon ^3+\frac{64 \epsilon }{\gamma }-162 \epsilon ^2-8 \epsilon -1\geq0$. Since $\epsilon\leq\frac{1}{1+\gamma}$, $\gamma\leq\frac{1}{\epsilon}-1$. Also, as $0<\gamma\leq1$ is assumed, the range of $\gamma$ can be expressed as $0<\gamma\leq \min(1,\frac{1}{\epsilon}-1)$. When $\frac{1}{6}<\epsilon\leq\frac{1}{2}$, $1<\frac{1}{\epsilon}-1$, so in this case, $0<\gamma\leq 1$. When $\frac{1}{2}\leq\epsilon<1$, so in this case, $0<\gamma\leq \frac{1}{\epsilon}-1$. Let $h(\gamma)=108\gamma\epsilon^3+\frac{64\epsilon}{\gamma}$, $\frac{\partial h(\gamma)}{\partial\gamma}=108\epsilon^3-\frac{64\epsilon}{\gamma^2}$. When $\gamma\leq\sqrt{\frac{64\epsilon}{18\epsilon^3}}$, $\frac{\partial h(\gamma)}{\partial\gamma}\geq0$, when $\gamma\geq\sqrt{\frac{64\epsilon}{18\epsilon^3}}$, $\frac{\partial h(\gamma)}{\partial\gamma}\leq0$, therefore, the minimum of $h(\gamma)$ must be when $\gamma$ is equal to the boundary point of the domain. When $\frac{1}{6}<\epsilon\leq\frac{1}{2}$, $0<\gamma\leq 1$, since $h(0)\rightarrow\infty$, $h(1)=108\epsilon^3+64\epsilon$, the minimum occurs at the boundary point $\gamma=1$, $108 \gamma  \epsilon ^3+\frac{64 \epsilon }{\gamma }-162 \epsilon ^2-8 \epsilon -1>108  \epsilon ^3+56 \epsilon-162 \epsilon ^2 -1$. Let $g(\epsilon)=108  \epsilon ^3+56 \epsilon-162 \epsilon ^2 -1$. $g'(\epsilon)=324 \epsilon ^2-324 \epsilon +56$, when $\epsilon\leq\frac{2}{9}$, $g'(\epsilon)\geq0$, when $\frac{2}{9}\leq\epsilon\leq\frac{1}{2}$, $g'(\epsilon)\leq0$, since $g(\frac{1}{6})=\frac{13}{3}$, $g(\frac{1}{2})=0$, so $g(\epsilon)\geq0$, $108 \gamma  \epsilon ^3+\frac{64 \epsilon }{\gamma }-162 \epsilon ^2-8 \epsilon -1\geq0$. When $\frac{1}{2}\leq\epsilon<1$, $0<\gamma\leq \frac{1}{\epsilon}-1$. Since $h(\frac{1}{\epsilon}-1)=108(\frac{1}{\epsilon}-1)\epsilon^3+\frac{64\epsilon}{\frac{1}{\epsilon}-1}$, $108 \gamma  \epsilon ^3+\frac{64 \epsilon }{\gamma }-162 \epsilon ^2-8 \epsilon -1>108\left(\frac{1}{\epsilon}-1\right)\epsilon^3+\frac{64\epsilon}{\frac{1}{\epsilon}-1}-162 \epsilon ^2-8 \epsilon -1=\frac{-108 \epsilon ^4+54 \epsilon ^3-18 \epsilon ^2+7 \epsilon +1}{\epsilon -1}$. Let $nu(\epsilon)=-108 \epsilon ^4+54 \epsilon ^3-18 \epsilon ^2+7 \epsilon +1$, then $nu'(\epsilon)=-432 \epsilon ^3+162 \epsilon ^2-36 \epsilon +7$, $nu''(\epsilon)=-1296 \epsilon ^2+324 \epsilon -36<0$. Since $nu'(\epsilon=\frac{1}{2})=-\frac{49}{2}<0$, $nu'(\epsilon)<0$. Also, $nu(\epsilon=\frac{1}{2})=0$, so $nu(\epsilon)\geq0$, $108 \gamma  \epsilon ^3+\frac{64 \epsilon }{\gamma }-162 \epsilon ^2-8 \epsilon -1\geq0$ is also valid. As a result, this simplified inequality is valid within the range of $\frac{1}{6}<\epsilon\leq\frac{1}{1+\gamma}$, when $0<\gamma\leq1$. Then, it validates $\frac{\partial \sup{\text{QA}}}{\partial\epsilon}\leq0$ for the same range of $\epsilon$ and $\gamma$. 

The first and second formulae, when $\epsilon=\frac{1}{6}$, are all equal to $\frac{1}{2} \left(\frac{\sqrt{\frac{\gamma }{4-\frac{\gamma }{2}}}}{\sqrt{2}}+\sqrt{\frac{5}{3}}\right)$. It follows that $\sup{\text{QA}}(\epsilon,\gamma)$ is continuous over $[0, \frac{1}{1+\gamma}]$. Hence, $\frac{\partial \sup{\text{QA}}}{\partial\epsilon}\leq0$ holds for the entire range $0\leq\epsilon\leq\frac{1}{1+\gamma}$, when $0\leq\gamma\leq1$, which leads to the assertion of this theorem.
\end{proof}

Let $\mathcal{P}_{\Upsilon}^{k}$ denote the set of all continuous distributions whose moments, from the first to the $k$th, are all finite. For a right-skewed distribution, it suffices to consider the upper bound. The monotonicity of $\sup_{P\in \mathcal{P}_{\Upsilon}^{2}}{\text{QA}}$  with respect to $\epsilon$ implies that the extent of any violations of the $\gamma$-orderliness, if $0\leq\gamma\leq1$, is bounded for any distribution with a finite second moment, e.g., for a right-skewed distribution in $\mathcal{P}_{\Upsilon}^{2}$, if $0\leq\epsilon_1\leq\epsilon_2\leq\epsilon_3\leq\frac{1}{1+\gamma}$, $\text{QA}_{\epsilon_2,\gamma} \geq \text{QA}_{\epsilon_3,\gamma} \geq \text{QA}_{\epsilon_1,\gamma}$, then $\text{QA}_{\epsilon_2,\gamma}$ will not be too far away from $\text{QA}_{\epsilon_1,\gamma}$, since $\sup_{P\in \mathcal{P}_{\Upsilon}^{2}}{\text{QA}_{\epsilon_1,\gamma}}>\sup_{P\in \mathcal{P}_{\Upsilon}^{2}}{\text{QA}_{\epsilon_2,\gamma}}>\sup_{P\in \mathcal{P}_{\Upsilon}^{2}}{\text{QA}_{\epsilon_3,\gamma}}$. Moreover, a stricter bound can be established for unimodal distributions according to Bernard et al. 's result \cite{bernard2020range}. The violation of $\nu$th $\gamma$-orderliness, when $\nu\geq2$, is also bounded, since the QA function is bounded, the $\nu$th $\gamma$-orderliness corresponds to the higher-order derivatives of the QA function with respect to $\epsilon$.

\section*{The Impact of $\gamma$-Orderliness on Weighted Inequalities}\label{sectionB}

Analogous to the $\gamma$-orderliness, the $\gamma$-trimming inequality for a right-skewed distribution is defined as $\forall 0\leq\epsilon_1\leq\epsilon_2\leq\frac{1}{1+\gamma}, \text{TM}_{\epsilon_1,\gamma} \geq \text{TM}_{\epsilon_2,\gamma}\text{.}$ $\gamma$-orderliness is a sufficient condition for the $\gamma$-trimming inequality, as proven in the SI Text. The next theorem shows a relation between the $\epsilon$,$\gamma$-quantile average and the $\epsilon$,$\gamma$-trimmed mean under the $\gamma$-trimming inequality, suggesting the $\gamma$-orderliness is not a necessary condition for the $\gamma$-trimming inequality.

\begin{theorem}\label{rmarkqi}For a distribution that is right-skewed and follows the $\gamma$-trimming inequality, it is asymptotically true that the quantile average is always greater or equal to the corresponding trimmed mean with the same $\epsilon$ and $\gamma$, for all $0\leq\epsilon\leq\frac{1}{1+\gamma}$.\end{theorem}\begin{proof}According to the definition of the $\gamma$-trimming inequality: $\forall 0\leq\epsilon\leq\frac{1}{1+\gamma}$, $\frac{1}{1-\epsilon-\gamma\epsilon+2\delta}\int_{\gamma\epsilon-\delta}^{1-\epsilon+\delta}{Q\left(u\right)du}\geq\frac{1}{1-\epsilon-\gamma\epsilon}\int_{\gamma\epsilon}^{1-\epsilon}{Q\left(u\right)du}$, where $\delta$ is an infinitesimal positive quantity. Subsequently, rewriting the inequality gives $\int_{\gamma\epsilon-\delta}^{1-\epsilon+\delta}{Q\left(u\right)du}-\frac{1-\epsilon-\gamma\epsilon+2\delta}{1-\epsilon-\gamma\epsilon}\int_{\gamma\epsilon}^{1-\epsilon}{Q\left(u\right)du}\geq0 \Leftrightarrow\int_{1-\epsilon}^{1-\epsilon+\delta}{Q\left(u\right)du}+\int_{\gamma\epsilon-\delta}^{\gamma\epsilon}{Q\left(u\right)du}-\frac{2\delta}{1-\epsilon-\gamma\epsilon}\int_{\gamma\epsilon}^{1-\epsilon}{Q\left(u\right)du}\geq0$. Since $\delta\rightarrow0^{+}$, $\frac{1}{2\delta}\left(\int_{1-\epsilon}^{1-\epsilon+\delta}{Q\left(u\right)du}+\int_{\gamma\epsilon-\delta}^{\gamma\epsilon}{Q\left(u\right)du}\right)=\frac{Q(\gamma\epsilon)+Q(1-\epsilon)}{2}\geq\frac{1}{1-\epsilon-\gamma\epsilon}\int_{\gamma\epsilon}^{1-\epsilon}{Q\left(u\right)du}$, the proof is complete.\end{proof}

An analogous result about the relation between the $\epsilon$,$\gamma$-trimmed mean and the $\epsilon$,$\gamma$-Winsorized mean can be obtained in the following theorem.\begin{theorem}\label{wti}For a right-skewed distribution following the $\gamma$-trimming inequality, asymptotically, the Winsorized mean is always greater or equal to the corresponding trimmed mean with the same $\epsilon$ and $\gamma$, for all $0\leq\epsilon\leq\frac{1}{1+\gamma}$, provided that $0\leq\gamma\leq1$. If assuming $\gamma$-orderliness, the inequality is valid for any non-negative $\gamma$. \end{theorem}\begin{proof}According to Theorem \ref{rmarkqi}, $\frac{Q\left(\gamma\epsilon\right)+Q\left(1-\epsilon\right)}{2}\geq\frac{1}{1-\epsilon-\gamma\epsilon}\int_{\gamma\epsilon}^{1-\epsilon}Q\left(u\right)du  \Leftrightarrow\gamma\epsilon\left(Q\left(\gamma\epsilon\right)+Q\left(1-\epsilon\right)\right)\geq(\frac{2\gamma\epsilon}{1-\epsilon-\gamma\epsilon})\int_{\gamma\epsilon}^{1-\epsilon}Q\left(u\right)du$. Then, if $0\leq\gamma\leq1$,$\left(1-\frac{1}{1-\epsilon-\gamma\epsilon}\right)\int_{\gamma\epsilon}^{1-\epsilon}Q\left(u\right)du+\gamma\epsilon\left(Q\left(\gamma\epsilon\right)+Q\left(1-\epsilon\right)\right)\geq0 \Rightarrow\int_{\gamma\epsilon}^{1-\epsilon}Q\left(u\right)du+\gamma\epsilon Q\left(\gamma\epsilon\right)+\epsilon Q\left(1-\epsilon\right)\geq\int_{\gamma\epsilon}^{1-\epsilon}Q\left(u\right)du+\gamma\epsilon\left(Q\left(\gamma\epsilon\right)+Q\left(1-\epsilon\right)\right)\geq\frac{1}{1-\epsilon-\gamma\epsilon}\int_{\gamma\epsilon}^{1-\epsilon}Q\left(u\right)du$, the proof of the first assertion is complete. The second assertion is established in Theorem 0.3. in the SI Text.\end{proof}

Replacing the TM in the $\gamma$-trimming inequality with WA forms the definition of the $\gamma$-weighted inequality. The $\gamma$-orderliness also implies the $\gamma$-Winsorization inequality when $0\leq\gamma\leq1$, as proven in the SI Text. The same rationale as presented in Theorem \ref{lst}, for a location-scale distribution characterized by a location parameter $\mu$ and a scale parameter $\lambda$, asymptotically, any $\mathrm{WA}(\epsilon,\gamma)$ can be expressed as $\lambda \mathrm{WA}_{0}(\epsilon,\gamma)+\mu$, where $\mathrm{WA}_{0}(\epsilon,\gamma)$ is an function of $Q_0(p)$ according to the definition of the weighted average. Adhering to the rationale present in Theorem \ref{lst}, for any probability distribution within a location-scale family, a necessary and sufficient condition for whether it follows the $\gamma$-weighted inequality is whether the family of probability distributions also adheres to the $\gamma$-weighted inequality. 

To construct weighted averages based on the $\nu$th $\gamma$-orderliness and satisfying the corresponding weighted inequality, when $0\leq\gamma\leq1$, let $\mathcal{B}_i=\int_{i\epsilon}^{(i+1)\epsilon}\text{QA}\left(u,\gamma\right)du$, $ka=k\epsilon+c$. From the $\gamma$-orderliness for a right-skewed distribution, it follows that, $-\frac{\partial\mathrm{QA}}{\partial\epsilon}\geq0\Leftrightarrow \forall 0\leq a\leq2a\leq\frac{1}{1+\gamma}, -\frac{\left(\mathrm{QA}\left(2a,\gamma\right)-\mathrm{QA}\left(a,\gamma\right)\right)}{a}\geq0  \Rightarrow \mathcal{B}_i-\mathcal{B}_{i+1}\geq0$, if $0\leq\gamma\leq1$. Suppose that $\mathcal{B}_i=\mathcal{B}_0$. Then, the $\epsilon$,$\gamma$-block Winsorized mean, is defined as \begin{align*}\mathrm{BWM}_{\epsilon,\gamma,n}\coloneqq\frac{1}{n}\left(\sum_{i=n\gamma\epsilon+1}^{\left(1-\epsilon\right)n}X_i+\sum_{i=n\gamma\epsilon+1}^{2n\gamma\epsilon+1}X_i+\sum_{i=\left(1-2\epsilon\right)n}^{\left(1-\epsilon\right)n}X_i\right)\text{,}\end{align*} which is double weighting the leftest and rightest blocks having sizes of $\gamma\epsilon n$ and $\epsilon n$, respectively. As a consequence of $\mathcal{B}_i-\mathcal{B}_{i+1}\geq0$, the $\gamma$-block Winsorization inequality is valid, provided that $0\leq\gamma\leq1$. The block Winsorized mean uses two blocks to replace the trimmed parts, not two single quantiles. The subsequent theorem provides an explanation for this difference.

\begin{theorem}\label{bwm}Asymptotically, for a right-skewed distribution following the $\gamma$-orderliness, the Winsorized mean is always greater than or equal to the corresponding block Winsorized mean with the same $\epsilon$ and $\gamma$, for all $0\leq\epsilon\leq\frac{1}{1+\gamma}$, provided that $0\leq\gamma\leq1$.\end{theorem}\begin{proof}From the definitions of BWM and WM, the statement necessitates $\int_{\gamma\epsilon}^{1-\epsilon}Q\left(u\right)du+\gamma\epsilon Q\left(\gamma\epsilon\right)+\epsilon Q\left(1-\epsilon\right)\geq\int_{\gamma\epsilon}^{1-\epsilon}Q\left(u\right)du+\int_{\gamma\epsilon}^{2\gamma\epsilon}Q\left(u\right)du+\int_{1-2\epsilon}^{1-\epsilon}Q\left(u\right)du \Leftrightarrow \gamma\epsilon Q\left(\gamma\epsilon\right)+\epsilon Q\left(1-\epsilon\right)\geq\int_{\gamma\epsilon}^{2\gamma\epsilon}Q\left(u\right)du+\int_{1-2\epsilon}^{1-\epsilon}Q\left(u\right)du$. Define $\text{WM}l(x)=Q\left(\gamma\epsilon\right)$ and $\text{BWM}l(x)=Q\left(x\right)$. In both functions, the interval for $x$ is specified as $[\gamma\epsilon,2\gamma\epsilon]$. Then, define $\text{WM}u(y)=Q\left(1-\epsilon\right)$ and $\text{BWM}u(y)=Q\left(y\right)$. In both functions, the interval for $y$ is specified as $[1-2\epsilon,1-\epsilon]$. The function $y: [\gamma\epsilon,2\gamma\epsilon] \rightarrow [1-2\epsilon,1-\epsilon]$ defined by $y(x)=1-\frac{x}{\gamma}$ is a bijection. $\text{WM}l(x)+\text{WM}u(y(x))=Q\left(\gamma\epsilon\right)+Q\left(1-\epsilon\right)\geq \text{BWM}l(x)+\text{BWM}u(y(x))=Q\left(x\right)+Q\left(1-\frac{x}{\gamma}\right)$ is valid for all $x\in[\gamma\epsilon,2\gamma\epsilon]$, according to the definition of $\gamma$-orderliness. Integration of the left side yields, $\int_{\gamma\epsilon}^{2\gamma\epsilon}\left(\mathrm{WM}l\left(u\right)+\mathrm{WM} u\left(y\left(u\right)\right)\right)du=\int_{\gamma\epsilon}^{2\gamma\epsilon}Q\left(\gamma\epsilon\right)du+\int_{y\left(\gamma\epsilon\right)}^{y\left(2\gamma\epsilon\right)}Q\left(1-\epsilon\right)du=\int_{\gamma\epsilon}^{2\gamma\epsilon}Q\left(\gamma\epsilon\right)du+\int_{1-2\epsilon}^{1-\epsilon}Q\left(1-\epsilon\right)du=\gamma\epsilon Q\left(\gamma\epsilon\right)+\epsilon Q\left(1-\epsilon\right)$, while integration of the right side yields $\int_{\gamma\epsilon}^{2\gamma\epsilon}\left(\mathrm{BWM}l\left(x\right)+\mathrm{BWM} u\left(y\left(x\right)\right)\right)dx=\int_{\gamma\epsilon}^{2\gamma\epsilon}Q\left(u\right)du+\int_{\gamma\epsilon}^{2\gamma\epsilon}Q\left(1-\frac{x}{\gamma}\right)dx=\int_{\gamma\epsilon}^{2\gamma\epsilon}Q\left(u\right)du+\int_{1-2\epsilon}^{1-\epsilon}Q\left(u\right)du$, which are the left and right sides of the desired inequality. Given that the upper limits and lower limits of the integrations are different for each term, the condition $0\leq\gamma\leq1$ is necessary for the desired inequality to be valid.

\end{proof} 
From the second $\gamma$-orderliness for a right-skewed distribution, $\frac{\partial^{2}\mathrm{QA}}{\partial^{2}\epsilon}\geq0\Rightarrow \forall 0\leq a\leq2a\leq3a\leq\frac{1}{1+\gamma}, \frac{1}{a}\left(\frac{\left(\mathrm{QA}\left(3a,\gamma\right)-\mathrm{QA} \left(2a,\gamma\right)\right)}{a}-\frac{\left(\mathrm{QA}\left(2a,\gamma\right)-\mathrm{QA} \left(a,\gamma\right)\right)}{a}\right)\geq0 \Rightarrow$ if $0\leq\gamma\leq1$, $\mathcal{B}_i -2\mathcal{B}_{i+1}+ \mathcal{B}_{i+2}\geq 0.$ $\text{SM}_{\epsilon}$ can thus be interpreted as assuming $\gamma=1$ and replacing the two blocks, $\mathcal{B}_i + \mathcal{B}_{i+2}$ with one block $2\mathcal{B}_{i+1}$. From the $\nu$th $\gamma$-orderliness for a right-skewed distribution, the recurrence relation of the derivatives naturally produces the alternating binomial coefficients, \begin{align*}(-1)^{\nu}\frac{\partial^{\nu}\text{QA}}{\partial\epsilon^{\nu}}\geq0\Rightarrow \forall 0\leq a\leq\ldots\leq(\nu+1)a\leq\frac{1}{1+\gamma},\\ \frac{(-1)^{\nu}}{a}\left(\frac{\frac{\mathrm{QA}\left(\nu a+a,\gamma\right) \ddots }{a}-\frac{\iddots\mathrm{QA}\left(2a,\gamma\right) }{a}}{a}-\frac{\frac{\mathrm{QA}\left(\nu a,\gamma\right) \ddots }{a}-\frac{\iddots\mathrm{QA}\left(a,\gamma\right) }{a}}{a}\right)\\\geq0\Leftrightarrow\frac{(-1)^{\nu}}{a^{\nu}}\left(\sum_{j=0}^{\nu}{\left(-1\right)^j\binom{\nu}{j}\mathrm{QA}\left(\left(\nu-j+1\right)a,\gamma\right)}\right)\geq0\\\Rightarrow \text{if } 0\leq\gamma\leq1,\sum_{j=0}^{\nu}{\left(-1\right)^j\binom{\nu}{j}\mathcal{B}_{i+j}}\geq 0.\end{align*}Based on the $\nu$th orderliness, the $\epsilon$,$\gamma$-binomial mean is introduced as \begin{align*}\text{BM}_{\nu,\epsilon,\gamma,n}\coloneqq\frac{1}{n}\left(\sum_{i=1}^{\frac{1}{2}\epsilon^{-1}(\nu+1)^{-1}}{\sum_{j=0}^{\nu}{\left(1-\left(-1\right)^j\binom{\nu}{j}\right)\mathfrak{B}_{i_j}}}\right)\text{,} \end{align*}\noindent where $\mathfrak{B}_{i_j}=\sum_{l=n\gamma\epsilon (j+(i-1)(\nu+1))+1}^{n\epsilon(j+(i-1)(\nu+1)+1)}{(X_l+X_{n-l+1})}$. If $\nu$ is not indicated, it defaults to $\nu=3$. Since the alternating sum of binomial coefficients equals zero, when $\nu\ll \epsilon^{-1}$ and $\epsilon\rightarrow0$, $\text{BM}\rightarrow\mu$. The solutions for the continuity of the breakdown point is the same as that in $\text{SM}$ and not repeated here. The equalities $\text{BM}_{\nu=1,\epsilon}=\text{BWM}_{\epsilon}$ and $\text{BM}_{\nu=2,\epsilon}=\text{SM}_{\epsilon,b=3}$ hold, when $\gamma=1$ and their respective $\epsilon$s are identical. Interestingly, the biases of the $\text{SM}_{\epsilon=\frac{1}{9},b=3}$ and the $\text{WM}_{\epsilon=\frac{1}{9}}$ are nearly indistinguishable in common asymmetric unimodal distributions such as Weibull, gamma, lognormal, and Pareto (SI Dataset S1). This indicates that their robustness to departures from the symmetry assumption is practically similar under unimodality, even though they are based on different orders of orderliness. If single quantiles are used, based on the second $\gamma$-orderliness, the stratified quantile mean can be defined as \begin{align*}\text{SQM}_{\epsilon,\gamma,n}\coloneqq4\epsilon\sum_{i=1}^{\frac{1}{4\epsilon}}\frac{1}{2}(\hat{Q}_{n}\left((2i-1)\gamma\epsilon\right)+\hat{Q}_{n}\left({1-(2i-1)\epsilon}\right))\text{,} \end{align*} $\text{SQM}_{\epsilon=\frac{1}{4}}$ is the Tukey’s midhinge \cite{tukey1977exploratory}. In fact, SQM is a subcase of SM when $\gamma=1$ and $b\rightarrow\infty$, so the solution for the continuity of the breakdown point, $\frac{1}{\epsilon}\text{ mod }4\neq0$, is identical. However, since the definition is based on the empirical quantile function, no decimal issues related to order statistics will arise. The next theorem explains another advantage.

\begin{theorem}\label{sqbm}For a right-skewed second $\gamma$-ordered distribution, asymptotically, $\text{SQM}_{\epsilon,\gamma}$ is always greater or equal to the corresponding $\text{BM}_{\nu=2,\epsilon,\gamma}$ with the same $\epsilon$ and $\gamma$, for all $0\leq\epsilon\leq\frac{1}{1+\gamma}$, if $0\leq\gamma\leq1$.\end{theorem}\begin{proof}For simplicity, suppose the order statistics of the sample are distributed into $\epsilon^{-1}\in\mathbb{N}$ blocks in the computation of both $\text{SQM}_{\epsilon,\gamma}$ and $\text{BM}_{\nu=2,\epsilon,\gamma}$. The computation of $\text{BM}_{\nu=2,\epsilon,\gamma}$ alternates between weighting and non-weighting, let ‘0’ denote the block assigned with a weight of zero and ‘1’ denote the block assigned with a weighted of one, the sequence indicating the weighted or non-weighted status of each block is: $0,1,0,0,1,0,\dots$. Let this sequence be denoted by $a_{\text{BM}_{\nu=2,\epsilon,\gamma}}(j)$, its formula is $a_{\text{BM}_{\nu=2,\epsilon,\gamma}}(j)= \left\lfloor\frac{j \bmod 3}{2}\right\rfloor$. Similarly, the computation of $\text{SQM}_{\epsilon,\gamma}$ can be seen as positioning quantiles ($p$) at the beginning of the blocks if $0<p<\frac{1}{1+\gamma}$, and at the end of the blocks if $p>\frac{1}{1+\gamma}$. The sequence of denoting whether each block's quantile is weighted or not weighted is: $0,1,0,1,0,1,\dots$. Let the sequence be denoted by $a_{\text{SQM}_{\epsilon,\gamma}}(j)$, the formula of the sequence is $a_{\text{SQM}_{\epsilon,\gamma}}(j)= j \bmod 2$. If pairing all blocks in $\text{BM}_{\nu=2,\epsilon,\gamma}$ and all quantiles in $\text{SQM}_{\epsilon,\gamma}$, there are two possible pairings of $a_{\text{BM}_{\nu=2}}(j)$ and $a_{\text{SQM}_{\epsilon,\gamma}}(j)$. One pairing occurs when $a_{\text{BM}_{\nu=2,\epsilon,\gamma}}(j)=a_{\text{SQM}_{\epsilon,\gamma}}(j)=1$, while the other involves the sequence $0,1,0$ from $a_{\text{BM}_{\nu=2,\epsilon,\gamma}}(j)$ paired with $1,0,1$ from $a_{\text{SQM}_{\epsilon,\gamma}}(j)$. By leveraging the same principle as Theorem \ref{bwm} and the second $\gamma$-orderliness (replacing the two quantile averages with one quantile average between them), the desired result follows.\end{proof}


The biases of $\text{SQM}_{\epsilon=\frac{1}{8}}$, which is based on the second orderliness with a quantile approach, are notably similar to those of $\text{BM}_{\nu=3,\epsilon=\frac{1}{8}}$, which is based on the third orderliness with a block approach, in common asymmetric unimodal distributions (Figure \ref{fig:Biasplot}).

\section*{Hodges–Lehmann Inequality and $\gamma$-$U$-Orderliness}\label{QQ99}The Hodges–Lehmann estimator stands out as a unique robust location estimator due to its definition being substantially dissimilar from conventional $L$-estimators, $R$-estimators, and $M$-estimators. In their landmark paper, \emph{Estimates of location based on rank tests}, Hodges and Lehmann \cite{hodges1963estimates} proposed two methods for computing the H-L estimator: the Wilcoxon score $R$-estimator and the median of pairwise means. The Wilcoxon score $R$-estimator is a location estimator based on signed-rank test, or $R$-estimator, \cite{hodges1963estimates} and was later independently discovered by Sen (1963) \cite{sen1963estimation}. However, the median of pairwise means is a generalized $L$-statistic and a trimmed $U$-statistic, as classified by Serfling in his novel conceptualized study in 1984 \cite{serfling1984generalized}. Serfling further advanced the understanding by generalizing the H-L kernel as $hl_{k}\left( x_1,\ldots,x_k\right)=\frac{1}{k}\sum_{i=1}^{k}x_i$, where $k\in\mathbb{N}$ \cite{serfling1984generalized}. Here, the weighted H-L kernel is defined as $whl_{k}\left( x_1,\ldots,x_k\right)=\frac{\sum_{i=1}^{k}{x_i \mathbf{w}_i}}{\sum_{i=1}^{k}{\mathbf{w}_i}}$, where $\mathbf{w}_i$s are the weights applied to each element.

By using the weighted H-L kernel and the $L$-estimator, it is now clear that the Hodges-Lehmann estimator is an $LL$-statistic, the definition of which is provided as follows: \begin{align*}LL_{k,\epsilon,\gamma,n}\coloneqq L_{\epsilon_0,\gamma,n}\left(\text{sort}\left(\left(whl_{k}\left(X_{N_1},\Compactcdots,X_{N_k}\right)\right)_{N=1}^{\binom{n}{k}}\right)\right),\end{align*} where $L_{\epsilon_0,\gamma,n}\left(Y\right)$ represents the $\epsilon_0$,$\gamma$-$L$-estimator that uses the sorted sequence, $\text{sort}\left(\left(whl_{k}\left(X_{N_1},\Compactcdots,X_{N_k}\right)\right)_{N=1}^{\binom{n}{k}}\right)$, as input. The upper asymptotic breakdown point of $LL_{k,\epsilon,\gamma}$ is $\epsilon=1-\left(1-\epsilon_0\right)^\frac{1}{k}$, as proven in REDS III \cite{unknown1}. There are two ways to adjust the breakdown point: either by setting $k$ as a constant and adjusting $\epsilon_0$, or by setting $\epsilon_0$ as a constant and adjusting $k$. In the above definition, $k$ is discrete, but the bootstrap method can be applied to ensure the continuity of $k$, also making the breakdown point continuous. Specifically, if $k\in\mathbb{R}$, let the bootstrap size be denoted by $b$, then first sampling the original sample $(1-k+\lfloor k \rfloor)b$ times with each sample size of $\lfloor k \rfloor$, and then subsequently sampling $(1-\lceil k \rceil+k)b$ times with each sample size of $\lceil k \rceil$, $(1-k+\lfloor k \rfloor)b\in\mathbb{N}$, $(1-\lceil k \rceil+k)b\in\mathbb{N}$. The corresponding kernels are computed separately, and the pooled sorted sequence is used as the input for the $L$-estimator. Let $\mathbf{S}_k$ represent the sorted sequence. Indeed, for any finite sample, $X$, when $k=n$, $\mathbf{S}_k$ becomes a single point, $whl_{k=n}\left(X_1,\ldots,X_n\right)$. When $\mathbf{w}_i=1$, the minimum of $\mathbf{S}_k$ is $\frac{1}{k}\sum_{i=1}^{k}X_{i}$, due to the property of order statistics. The maximum of $\mathbf{S}_k$ is $\frac{1}{k}\sum_{i=1}^{k}X_{n-i+1}$. The monotonicity of the order statistics implies the monotonicity of the extrema with respect to $k$, i.e., the support of $\mathbf{S}_k$ shrinks monotonically. For unequal $\mathbf{w}_i$s, the shrinkage of the support of $\mathbf{S}_k$ might not be strictly monotonic, but the general trend remains, since all $LL$-statistics converge to the same point, as $k\to n$. Therefore, if $\frac{\sum_{i=1}^{n}{X_i\mathbf{w}_i}}{\sum_{i=1}^{n}\mathbf{w}_i}$ approaches the population mean when $n\to\infty$, all $LL$-statistics based on such consistent kernel function approach the population mean as $k\to\infty$. For example, if $whl_{k}=\text{BM}_{\nu,\epsilon_k,n=k}$, $\nu\ll \epsilon_k^{-1}$, $\epsilon_k\rightarrow0$, such kernel function is consistent. These cases are termed the $LL$-mean ($\text{LLM}_{k,\epsilon,\gamma,n}$). By substituting the $\text{WA}_{\epsilon_0,\gamma,n}$ for the $L_{\epsilon_0,\gamma,n}$ in $LL$-statistic, the resulting statistic is referred to as the weighted $L$-statistic ($\text{WL}_{k,\epsilon,\gamma,n}$). The case having a consistent kernel function is termed as the weighted $L$-mean ($\text{WLM}_{k,\epsilon,\gamma,n}$). The $w_i=1$ case of $\text{WLM}_{k,\epsilon,\gamma,n}$ is termed the weighted Hodges-Lehmann mean ($\text{WHLM}_{k,\epsilon,\gamma,n}$). The $\text{WHLM}_{k=1,\epsilon,\gamma,n}$ is the weighted average. If $k\geq2$ and the $\text{WA}$ in WHLM is set as $\text{TM}_{\epsilon_0}$, it is called the trimmed H-L mean (Figure \ref{fig:Biasplot}, $k=2$, $\epsilon_0=\frac{15}{64}$). The $\text{THLM}_{k=2,\epsilon,\gamma=1,n}$ appears similar to the Wilcoxon's one-sample statistic investigated by Saleh in 1976 \cite{ehsanes1976hodges}, which involves first censoring the sample, and then computing the mean of the number of events that the pairwise mean is greater than zero. The $\text{THLM}_{k=2,\epsilon=1-\left(1-\frac{1}{2}\right)^\frac{1}{2},\gamma=1,n}$ is the Hodges-Lehmann estimator, or more generally, a special case of the median Hodges-Lehmann mean ($m\text{HLM}_{k,n}$). $m\text{HLM}_{k,n}$ is asymptotically equivalent to the $\text{MoM}_{k,b=\frac{n}{k}}$ as discussed previously, Therefore, it is possible to define a series of location estimators, analogous to the WHLM, based on MoM. For example, the $\gamma$-median of means, $\gamma m\text{oM}_{k,b=\frac{n}{k},n}$, is defined by replacing the median in $\text{MoM}_{k,b=\frac{n}{k},n}$ with the $\gamma$-median.


The $hl_k$ kernel distribution, denoted as $F_{hl_k}$, can be defined as the probability distribution of the sorted sequence, $\text{sort}\left(\left(hl_{k}\left(X_{N_1},\Compactcdots,X_{N_k}\right)\right)_{N=1}^{\binom{n}{k}}\right)$. For any real value $y$, the cdf of the $hl_k$ kernel distribution is given by: $F_{h_{\mathbf{k}}} (y) =\mathbb{P}(Y_i \leq y)$, where $Y_i$ represents an individual element from the sorted sequence. The overall $hl_k$ kernel distributions possess a two-dimensional structure, encompassing $n$ kernel distributions with varying $k$ values, from $1$ to $n$, where one dimension is inherent to each individual kernel distribution, while the other is formed by the alignment of the same percentiles across all kernel distributions. As $k$ increases, all percentiles converge to $\bar{X}$, leading to the concept of $\gamma$-$U$-orderliness:\begin{align*}(\forall k_2\geq k_1\geq 1, \text{QHLM}_{k_2,\epsilon=1-\left(\frac{\gamma}{1+\gamma}\right)^\frac{1}{k_2},\gamma} \geq \text{QHLM}_{k_1,\epsilon=1-\left(\frac{\gamma}{1+\gamma}\right)^\frac{1}{k_1},\gamma})\vee\\(\forall k_2\geq k_1\geq 1, \text{QHLM}_{k_2,\epsilon=1-\left(\frac{\gamma}{1+\gamma}\right)^\frac{1}{k_2},\gamma} \leq \text{QHLM}_{k_1,\epsilon=1-\left(\frac{\gamma}{1+\gamma}\right)^\frac{1}{k_1},\gamma}),\end{align*} where $\text{QHLM}_{k}$ sets the WA in WHLM as quantile average, with $\gamma$ being constant. The direction of the inequality depends on the relative magnitudes of $\text{QHLM}_{k=1,\epsilon,\gamma}=\text{QA}$ and $\text{QHLM}_{k=\infty,\epsilon,\gamma}=\mu$. The Hodges-Lehmann inequality can be defined as a special case of the $\gamma$-$U$-orderliness when $\gamma=1$ and quantile average is median. When $\gamma\in \{0, \infty\}$, the quantile average is $\gamma$-median, the $\gamma$-$U$-orderliness is valid for any distribution as previously shown. If $\gamma\notin\{0, \infty\}$, analytically proving the validity of the $\gamma$-$U$-orderliness for a parametric distribution is pretty challenging. As an example, the $hl_{2}$ kernel distribution has a probability density function $f_{hl_{2}}(x)=\int_{0}^{2x}2f\left(t\right)f\left(2x-t\right)dt$ (a result after the transformation of variables); the support of the original distribution is assumed to be $[0,\infty)$ for simplicity. The expected value of the H-L estimator is the positive solution of $\int_{0}^{\text{H-L}}\left(f_{hl_{2}}(s)\right)ds=\frac{1}{2}$. For the exponential distribution, $f_{hl_{2},exp}(x)=4 \lambda ^{-2} x e^{-2 \lambda^{-1} x}$, $\lambda$ is a scale parameter, $E[\text{H-L}]=\frac{-W_{-1}\left(-\frac{1}{2 e}\right)-1}{2}\lambda\approx0.839\lambda$, where $W_{-1}$ is a branch of the Lambert $W$ function which cannot be expressed in terms of elementary functions. However, when the quantile average is $\gamma$-median, the violation of the $\gamma$-$U$-orderliness is bounded under certain assumptions, as shown below. 

\begin{theorem}\label{pomb1} For any distribution with a finite second central moment, $\sigma^2$, the following concentration bound can be established for the $\gamma$-median of means,\begin{align*}\mathbb{P}\left(\gamma m\text{oM}_{k,b=\frac{n}{k},n}-\mu>\frac{t\sigma}{\sqrt k}\right)\le e^{-\frac{2n}{k}\left(\frac{1}{1+\gamma}-\frac{1}{k+t^2}\right)^2}.\end{align*} \end{theorem}\begin{proof}Denote the mean of each block as $\widehat{\mu_i}$, $1\leq i\leq b$. Observe that the event $\left\{\gamma m\text{oM}_{k,b=\frac{n}{k},n}-\mu>\frac{t\sigma}{\sqrt k}\right\}$ necessitates the condition that there are at least $b(1-\frac{\gamma}{1+\gamma})$ of $\widehat{\mu_i}$s larger than $\mu$ by more than $\frac{t\sigma}{\sqrt k}$, i.e., $\left\{\gamma m\text{oM}_{k,b=\frac{n}{k},n}-\mu>\frac{t\sigma}{\sqrt k}\right\}\subset\left\{\sum_{i=1}^{b}\mathbf{1}_{\left(\widehat{\mu_i}-\mu\right)>\frac{t\sigma}{\sqrt k}}\geq b\left(1-\frac{\gamma}{1+\gamma}\right)\right\}$, where $\mathbf{1}_A$ is the indicator of event $A$. Assuming a finite second central moment, $\sigma^2$, it follows from one-sided Chebeshev’s inequality that $\mathbb{E}\left(\mathbf{1}_{\left(\widehat{\mu_i}-\mu\right)>\frac{t\sigma}{\sqrt k}}\right)=\mathbb{P}\left(\left(\widehat{\mu_i}-\mu\right)>\frac{t\sigma}{\sqrt k}\right)\leq\frac{\sigma^2}{k\sigma^2+t^2\sigma^2}$.

Given that $\mathbf{1}_{\left(\widehat{\mu_i}-\mu\right)>\frac{t\sigma}{\sqrt k}}\in[0,1]$ are independent and identically distributed random variables, according to the aforementioned inclusion relation, the one-sided Chebeshev’s inequality and the one-sided Hoeffding's inequality, $\mathbb{P}\left(\gamma m\text{oM}_{k,b=\frac{n}{k},n}-\mu>\frac{t\sigma}{\sqrt k}\right)\leq \mathbb{P}\left(\sum_{i=1}^{b}\mathbf{1}_{\left(\widehat{\mu_i}-\mu\right)>\frac{t\sigma}{\sqrt k}}\geq b\left(1-\frac{\gamma}{1+\gamma}\right)\right)=\mathbb{P}\left(\frac{1}{b}\sum_{i=1}^{b}\left(\mathbf{1}_{\left(\widehat{\mu_i}-\mu\right)>\frac{t\sigma}{\sqrt k}}-\mathbb{E}\left(\mathbf{1}_{\left(\widehat{\mu_i}-\mu\right)>\frac{t\sigma}{\sqrt k}}\right)\right)\geq\right. \nonumber \\ \qquad \left.  \left(1-\frac{\gamma}{1+\gamma}\right)- \mathbb{E}\left(\mathbf{1}_{\left(\widehat{\mu_i}-\mu\right)>\frac{t\sigma}{\sqrt k}}\right)\right)\leq e^{-2b \left(\left(1-\frac{\gamma}{1+\gamma}\right)- \mathbb{E}\left(\mathbf{1}_{\left(\widehat{\mu_i}-\mu\right)>\frac{t\sigma}{\sqrt k}}\right)\right)^2}\leq e^{-2b\left(1-\frac{\gamma}{1+\gamma}-\frac{\sigma^2}{k\sigma^2+t^2\sigma^2}\right)^2}=e^{-2b\left(\frac{1}{1+\gamma}-\frac{1}{k+t^2}\right)^2}$. \end{proof}

\begin{theorem}\label{pomb} Let $B(k,\gamma,t,n)=e^{-\frac{2n}{k}\left(\frac{1}{1+\gamma}-\frac{1}{k+t^2}\right)^2}$. If $n\in\mathbb{N}$, $\gamma\geq0$, $0\leq t^2<\gamma +1$, and $\gamma -t^2+1\leq k\leq \frac{1}{2} \sqrt{9 \gamma ^2+18 \gamma -8 \gamma  t^2-8 t^2+9}+\frac{1}{2} \left(3 \gamma -2 t^2+3\right)$, $B$ is monotonic decreasing with respect to $k$.\end{theorem}\begin{proof} Since $\frac{\partial B}{\partial k}=\left(\frac{2 n \left(\frac{1}{\gamma +1}-\frac{1}{k+t^2}\right)^2}{k^2}-\frac{4 n \left(\frac{1}{\gamma +1}-\frac{1}{k+t^2}\right)}{k \left(k+t^2\right)^2}\right)\\e^{-\frac{2 n \left(\frac{1}{\gamma +1}-\frac{1}{k+t^2}\right)^2}{k}}$ and $n\in\mathbb{N}$, $\frac{\partial B}{\partial k}\leq 0\Leftrightarrow \frac{2 n \left(\frac{1}{\gamma +1}-\frac{1}{k+t^2}\right)^2}{k^2}-\frac{4 n \left(\frac{1}{\gamma +1}-\frac{1}{k+t^2}\right)}{k \left(k+t^2\right)^2}\leq0 \Leftrightarrow \frac{2 n \left(-\gamma +k+t^2-1\right) \left(k^2-3 (\gamma +1) k+2 k t^2+t^2 \left(-\gamma +t^2-1\right)\right)}{(\gamma +1)^2 k^2 \left(k+t^2\right)^3}\leq 0 \Leftrightarrow \left(-\gamma +k+t^2-1\right) \left(k^2-3 (\gamma +1) k+2 k t^2+t^2 \left(-\gamma +t^2-1\right)\right)\\\leq 0$. When the factors are expanded, it yields a cubic inequality in terms of $k$: $k^3+k^2 \left(3 t^2-4 (\gamma +1)\right)+3 k \left(\gamma -t^2+1\right)^2+t^2 \left(\gamma -t^2+1\right)^2\leq 0$. Assuming $0\leq t^2<\gamma +1$ and $\gamma\geq0$, using the factored form and subsequently applying the quadratic formula, the inequality is valid if $\gamma -t^2+1\leq k\leq \frac{1}{2} \sqrt{9 \gamma ^2+18 \gamma -8 \gamma  t^2-8 t^2+9}+\frac{1}{2} \left(3 \gamma -2 t^2+3\right)$.\end{proof}
 
Let $X$ be a random variable and $\bar{Y} = \frac{1}{k}(Y_1 +\dots + Y_k)$ be the average of $k$ independent, identically distributed copies of $X$. Applying the variance operation gives: $\text{Var}(\bar{Y}) = \text{Var}\left(\frac{1}{k}(Y_1 + \dots + Y_k)\right)= \frac{1}{k^2}(\text{Var}(Y_1) +\dots + \text{Var}(Y_k))= \frac{1}{k^2}(k\sigma^2) = \frac{\sigma^2}{k},$ since the variance operation is a linear operator for independent variables, and the variance of a scaled random variable is the square of the scale times the variance of the variable, i.e., $\text{Var}(cX) = E[(cX - E[cX])^2] = E[(cX - cE[X])^2]=E[c^2(X - E[X])^2] = c^2 E[((X) - E[X])^2]=c^2 \text{Var}(X)$. Thus, the standard deviation of the $hl_k$ kernel distribution, asymptotically, is $\frac{\sigma}{\sqrt k}$. By utilizing the asymptotic bias bound of any quantile for any continuous distribution with a finite second central moment, $\sigma^2$ \cite{li2018worst}, a conservative asymptotic bias bound of $\gamma m{\mathrm{oM}}_{k,b=\frac{n}{k}}$ can be established as $\gamma m{\mathrm{oM}}_{k,b=\frac{n}{k}}-\mu\leq\sqrt{\frac{\frac{\gamma}{1+\gamma}}{1-\frac{\gamma}{1+\gamma}}}\sigma_{hl_k}=\sqrt\frac{\gamma}{k} \sigma$. That implies in Theorem \ref{pomb1}, $t<\sqrt\gamma$, so when $\gamma=1$, the upper bound of $k$, subject to the monotonic decreasing constraint, is $2+\sqrt{5}<\frac{1}{2}\sqrt{9+18-8t^2-8t^2+9}+\frac{1}{2}\left(3-2t^2+3\right)\leq6$, the lower bound is $1< 2-t^2 \leq2$. These analyses elucidate a surprising result: although the conservative asymptotic bound of ${\mathrm{MoM}}_{k,b=\frac{n}{k}}$ is monotonic with respect to $k$, its concentration bound is optimal when $k\in (2+\sqrt{5},6]$.

Then consider the structure within each individual $hl_k$ kernel distribution. The sorted sequence $\mathbf{S}_k$, when $k=n-1$, has $n$ elements and the corresponding $hl_{k}$ kernel distribution can be seen as a location-scale transformation of the original distribution, so the corresponding $hl_{k}$ kernel distribution is $\nu$th $\gamma$-ordered if and only if the original distribution is $\nu$th $\gamma$-ordered according to Theorem \ref{lst}. Analytically proving other cases is challenging. For example, $f'_{hl_{2}}(x)=4f\left(2x\right)f\left(0\right)+\int_{0}^{2x}4f\left(t\right)f'\left(2x-t\right)dt$, the strict negative of $f'_{hl_{2}}(x)$ is not guaranteed if just assuming $f'(x)<0$, so, even if the original distribution is monotonic decreasing, the $hl_2$ kernel distribution might be non-monotonic. Also, unlike the pairwise difference distribution, if the original distribution is unimodal, the pairwise mean distribution might be non-unimodal, as demonstrated by a counterexample given by Chung in 1953 and mentioned by Hodges and Lehmann in 1954 \cite{hodges1954matching,chung1953lois}. Theorem \ref{qabbb} implies that the violation of $\nu$th $\gamma$-orderliness within the $hl_k$ kernel distribution is also bounded, and the bound monotonically shrinks as $k$ increases because the bound is in unit of the standard deviation of the $hl_k$ kernel distribution. If all $hl_k$ kernel distributions are $\nu$th $\gamma$-ordered and the distribution itself is $\nu$th $\gamma$-ordered and $\gamma$-$U$-ordered, then the distribution is called $\nu$th $\gamma$-$U$-ordered. The following theorem highlights the significance of symmetric distribution. 

\begin{theorem}\label{whlk1} Any symmetric distribution is $\nu$th $U$-ordered. \end{theorem} 
\begin{proof} 

A random variable is symmetric about zero if and only if its characteristic function is real valued. Since the characteristic function of the average of $k$ independent, identically distributed random variables is the product of the $k$th root of their individual characteristic functions : $\varphi_{\bar{Y}}(t) = \prod_{r=1}^k (\varphi_{Y_r}(t))^{\frac{1}{k}}$, \( \bar{Y} \) is symmetric. The conclusion follows immediately from the definition of $\nu$th $U$-orderliness and Theorem \ref{lst}, \ref{gsdqi}, and \ref{sdqi}.

\end{proof}

\begin{figure*}
\centering
\includegraphics[width=1\linewidth]{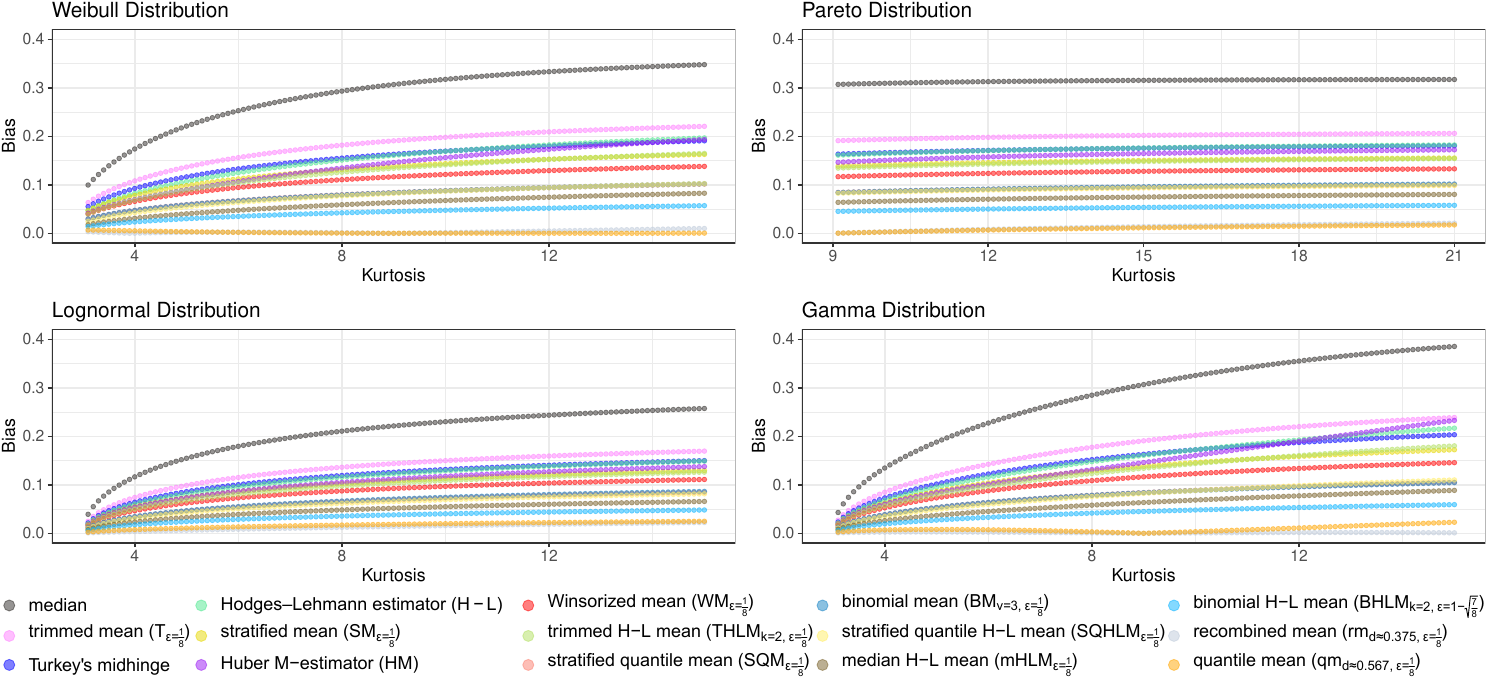}
\caption{Standardized biases (with respect to $\mu$) of fifteen robust location estimates (including two parametric estimators from REDS III \cite{li2024robust} for better comparison) on large quasi-random samples in four two-parameter right skewed unimodal distributions, as a function of the kurtosis. The methods are described in the SI Text.}
\label{fig:Biasplot}
\end{figure*}

The succeeding theorem shows that the $whl_k$ kernel distribution is invariably a location-scale distribution if the original distribution belongs to a location-scale family with the same location and scale parameters. \begin{theorem}\label{whlk} $whl_k\left(x_1=\lambda x_1+\mu,\ldots,x_k=\lambda x_k+\mu\right)=\lambda whl_k\left(x_1,\ldots,x_k\right)+\mu$. \end{theorem} \begin{proof}$whl_k\left(x_1=\lambda x_1+\mu,\Compactcdots,x_k=\lambda x_k+\mu\right)=\frac{\sum_{i=1}^{k}\left(\lambda x_i+\mu\right)w_i}{\sum_{i=1}^{k}w_i}=\frac{\sum_{i=1}^{k}{\lambda x_iw_i}+\sum_{i=1}^{k}\mu w_i}{\sum_{i=1}^{k}w_i}=\lambda\frac{\sum_{i=1}^{k}{x_iw_i}}{\sum_{i=1}^{k}w_i}+\frac{\sum_{i=1}^{k}{\mu w_i}}{\sum_{i=1}^{k}w_i}=\lambda\frac{\sum_{i=1}^{k}{x_iw}_i}{\sum_{i=1}^{k}w_i}+\mu=\lambda whl_k\left(x_1,\Compactcdots,x_k\right)+\mu$.\end{proof} According to Theorem \ref{whlk}, the $\gamma$-weighted inequality for a right-skewed distribution can be modified as $\forall 0\leq\epsilon_{0_1}\leq\epsilon_{0_2}\leq\frac{1}{1+\gamma}, \text{WLM}_{k,\epsilon=1-\left(1-\epsilon_{0_1}\right)^\frac{1}{k},\gamma} \geq \text{WLM}_{k,\epsilon=1-\left(1-\epsilon_{0_2}\right)^\frac{1}{k},\gamma}\text{,}$ which holds the same rationale as the $\gamma$-weighted inequality defined in the last section. If the $\nu$th $\gamma$-orderliness is valid for the $whl_k$ kernel distribution, then all results in the last section can be directly implemented. From that, the binomial H-L mean (set the WA as BM) can be constructed (Figure \ref{fig:Biasplot}), while its maximum breakdown point is $\approx0.065$ if $\nu=3$. A comparison of the biases of $\mathrm{STM}_{\epsilon=\frac{1}{8}}$, $\mathrm{SWM}_{\epsilon=\frac{1}{8}}$, $\mathrm{BWM}_{\epsilon=\frac{1}{8}}$, $\text{BM}_{\nu=2, \epsilon=\frac{1}{8}}$, $\text{BM}_{\nu=3, \epsilon=\frac{1}{8}}$, $\text{SQM}_{\epsilon=\frac{1}{8}}$, $\text{THLM}_{k=2,\epsilon=\frac{1}{8}}$, $\text{WiHLM}_{k=2,\epsilon=\frac{1}{8}}$  (Winsorized H-L mean), $\text{SQHLM}_{k=\frac{2 \ln (2)-\ln (3)}{3 \ln (2)-\ln (7)},\epsilon=\frac{1}{8}}$, $m\text{HLM}_{k=\frac{\ln (2)}{3 \ln (2)-\ln (7)}, \epsilon=\frac{1}{8}}$, $\text{THLM}_{k=5,\epsilon=\frac{1}{8}}$, and $\text{WiHLM}_{k=5,\epsilon=\frac{1}{8}}$ is appropriate (Figure \ref{fig:Biasplot}, SI Dataset S1), given their same breakdown points, with $m\text{HLM}_{k=\frac{\ln (2)}{3 \ln (2)-\ln (7)}, \epsilon=\frac{1}{8}}$ exhibiting the smallest biases. Another comparison among the H-L estimator, the trimmed mean, and the Winsorized mean, all with the same breakdown point, yields the same result that the H-L estimator has the smallest biases (SI Dataset S1). This aligns with Devroye et al. (2016) and Laforgue, Clemencon, and Bertail (2019)'s seminal works that $\text{MoM}_{k,b=\frac{n}{k}}$ and $\text{MoRM}_{k,b,n}$ are nearly optimal with regards to concentration bounds for heavy-tailed distributions \cite{devroye2016sub,laforgue2019medians}. 

In 1958, Richtmyer introduced the concept of quasi-Monte Carlo simulation that utilizes low-discrepancy sequences, resulting in a significant reduction in computational expenses for large sample simulation \cite{richtmyer1958non}. Among various low-discrepancy sequences, Sobol sequences are often favored in quasi-Monte Carlo methods \cite{sobol1967distribution}. Building upon this principle, in 1991, Do and Hall extended it to bootstrap and found that the quasi-random approach resulted in lower variance compared to other bootstrap Monte Carlo procedures \cite{Do1991QuasirandomRF}. By using a deterministic approach, the variance of $m\text{HLM}_{k,n}$ is much lower than that of $\text{MoM}_{k,b=\frac{n}{k}}$ (SI Dataset S1), when $k$ is small. This highlights the superiority of the median Hodges-Lehmann mean over the median of means, as it not only can provide an accurate estimate for moderate sample sizes, but also allows the use of quasi-bootstrap, where the bootstrap size can be adjusted as needed.

\matmethods{The Monte Carlo studies were conducted using the R programming language (version 4.3.1) with the following libraries: randtoolbox \cite{dutang2015package}, Rcpp \cite{eddelbuettel2011rcpp}, Rfast \cite{papadakis2023package}, matrixStats \cite{bengtsson2023package}, foreach \cite{weston2019foreach}, and doParallel \cite{weston2015getting}. Methods and codes described in the robust statistics book by Hampel, Ronchetti, Rousseeuw, and Stahel (2011) \cite{hampel2011robust}, the book by Maronna, Martin, Yohal, and Salibián-Barrera (2019) \cite{maronna2019robust}, and the book by Mair and Wilcox (2020) \cite{mair2020robust} contributed to the initial preparation of this manuscript for comparing current popular robust methods. However, due to their incompatibility, they were not included in the final analysis. The robust location estimates presented in Figure \ref{fig:Biasplot} and SI Dataset S1 were obtained using large quasi-random samples \cite{richtmyer1958non,sobol1967distribution} with sample size 3.686 million for the Weibull, gamma, Pareto, and lognormal distributions within specified kurtosis ranges as shown in Figure \ref{fig:Biasplot} to study the large sample performance. The standard errors of these estimators were computed by approximating the sampling distribution using 1000 pseudorandom samples of size $n = 5184$ for these distribution and the generalized Gaussian distributions with the parameter settings detailed in the SI Text. ggplot2 \cite{ggplot21} was used to generate Figure 1. ChatGPT, an AI language model developed by OpenAI, was used to enhance the grammar of this paper. To deduce and verify complex mathematical expressions, both Wolfram Alpha and ChatGPT were utilized. 
}

\showmatmethods{} 

\subsection*{Data and Software Availability} Data for Figure \ref{fig:Biasplot} are given in SI Dataset S1. All codes have been deposited in \href {https://github.com/johon-lituobang/REDS_Mean}{github.com/johon-lituobang/REDS}.
\acknow{I sincerely acknowledge the insightful comments from Peter Bickel, which considerably elevating the lucidity and merit of this paper. I am also grateful to Ruodu Wang for pointing out important mistakes regarding the $\gamma$-symmetric distribution.}

\showacknow{} 

\bibliography{pnas-sample}

\end{document}